\documentclass{amsart}

\usepackage{amssymb, amsmath,  amsfonts,amscd} 
\usepackage{color}
\usepackage{color}
\usepackage{xcolor}
\usepackage{graphicx}
\usepackage{array}
\usepackage{mathtools}
\usepackage{multirow}
\usepackage{hyperref}
\hypersetup{colorlinks}
\definecolor{darkred}{rgb}{0.5,0,0}
\definecolor{darkgreen}{rgb}{0,0.5,0}
\definecolor{darkblue}{rgb}{0,0,0.7}
\definecolor{orchid}{rgb}{.75,.24,1}
\definecolor{chocolate}{rgb}{.6,.2,0}
\definecolor{teal}{rgb}{.094,.655,.710}
\definecolor{comp_teal}{rgb}{0,.502,.502}
\hypersetup{colorlinks,
linkcolor=darkblue,
filecolor=darkgreen,
urlcolor=darkred,
citecolor=darkblue}
\usepackage{tikz}
\usetikzlibrary{matrix,arrows}
\usepackage{enumitem}

\newcommand\coolunder[2]{\mathrlap{\smash{\underbrace{\phantom{%
    \begin{matrix} #2 \end{matrix}}}_{\mbox{$#1$}}}}#2}

\newcommand{\+}[1]{\ensuremath{\mathbf{#1}}}
\newcommand{\C}{{\mathbb C}}
\newcommand{\Z}{{\mathbb Z}}
\renewcommand{\P}{{\mathbb P}}

\newcommand{\bull}{{\scriptscriptstyle \bullet}}

\newcommand{\cO}{{\mathcal O}}

\newcommand{\comment}[1]{}
\newcommand{\type}{\mathfrak{t}}
\newcommand{\op}{\text{op}}

\newcommand{\sym}{\mathfrak{s}}

\DeclarePairedDelimiter{\ceil}{\lceil}{\rceil}
\DeclarePairedDelimiter{\floor}{\lfloor}{\rfloor}
\renewcommand{\emptyset}{\varnothing}
\renewcommand{\tilde}{\widetilde}

\newtheorem{thm}{Theorem}[section]
\newtheorem{lemma}[thm]{Lemma}
\newtheorem{prop}[thm]{Proposition}
\newtheorem{cor}[thm]{Corollary}

\theoremstyle{definition}

\newtheorem{example}[thm]{Example}

\newtheorem{obs}[thm]{Observation}

\begin{document}
\pagestyle{plain}
\title{Triple Intersection Formulas for Isotropic Grassmannians}
\date{March 7, 2014}
\author{Vijay Ravikumar}
\address{Mathematics Department\\ Tata Institute of Fundamental Research\\ Mumbai, India}
\email{vravi@math.tifr.res.in}
\subjclass[2010]{Primary 14N15; Secondary 19E08, 14M15}

\begin{abstract}
Let $X$ be an isotropic Grassmannian of type $B$, $C$, or $D$.
In this paper we calculate $K$-theoretic Pieri-type triple intersection numbers for $X$: 
that is, the sheaf Euler characteristic of the triple intersection of two arbitrary Schubert varieties
and a special Schubert variety in general position.
We do this by determining explicit equations for the projected Richardson variety
corresponding to the two arbitrary Schubert varieties,
and show that it is a complete intersection in projective space.
The $K$-theoretic Pieri coefficients are alternating sums of these triple intersection numbers, 
and we hope they will lead to positive Pieri
formulas for isotropic Grassmannians.
\comment{
that is, the 
sheaf euler characteristic of the product of a special Schubert class with 
two arbitrary Schubert classes in the Grothendieck ring of $X$.
We do this by determining 
explicit equations for the projected Richardson variety
corresponding to the two arbitrary Schubert classes, which
is shown to be a complete intersection in projective space.
The $K$-theoretic Pieri coefficients are alternating sums of these triple intersection numbers, 
and we hope they will lead to positive Pieri
formulas for isotropic Grassmannians.
}
\end{abstract}

\maketitle
\tableofcontents

\section{Introduction}\label{S:intro}
When studying the ordinary cohomology of an (isotropic) Grassmannian, a \emph{triple intersection number} refers
to the number of intersection points of three Schubert
varieties in general position. 
By convention, this number is zero when the triple intersection has positive dimension.
Algebraically this number is given as the pushforward
of the product of three Schubert classes
to the the cohomology ring of a single point.

Given three Schubert varieties in general position, let $Z$ denote their scheme theoretic triple intersection.
The corresponding \emph{$K$-theoretic triple intersection number} is the sheaf Euler characteristic of $Z$; that is,
the pushforward of the product of the three Schubert classes
to the Grothendieck ring of a point.  We denote this number $\chi(Z)$.
If $Z$ is finite,
then just as in cohomology, $\chi(Z)$ is equal to the number of points in $Z$
(since these finitely many points are reduced by Kleiman's transversality theorem \cite{kleiman:transversality}).
If $Z$ has positive dimension however, then $\chi(Z)$ can be a nonzero (and possibly negative) integer.

In either setting, the triple intersection numbers
determine the structure contants for
multiplication with respect to the Schubert basis.
These structure constants are known as \emph{Littlewood-Richardson
coefficients} and in ordinary cohomology 
they are equal to triple intersection numbers.
In $K$-theory however, the Littlewood-Richardson coefficients are
alternating sums of triple intersection numbers.

An arbitrary Schubert class can be written as an integer polynomial in
certain \emph{special Schubert classes}, which (in cohomology)
are closely related to the Chern classes of the tautological quotient bundle on the Grassmannian
in question.
A triple intersection number is said to be of \emph{Pieri-type} if
one of the three Schubert classes is a special Schubert class.
Similarly, a \emph{Pieri coefficient} refers to a Littlewood-Richardson
coefficient occuring in the product of an arbitrary Schubert class
and a special Schubert class.

In this paper, we
determine $K$-theoretic Pieri-type triple intersection numbers
for all isotropic Grassmannians of types $B$, $C$, and $D$. 
\comment{The Pieri coefficients are alternating sums of these numbers.}
Our results generalize \cite{buch:cominuscule}, in which
similar calculations are carried out for the cominuscule Grassmannians;
that is for the type $A$ Grassmannian $Gr(m,\C^N)$, 
the maximal odd orthogonal Grassmannian $OG(m,\C^{2m+1})$, and the
Lagrangian Grassmannian $LG(m,\C^{2m})$.

\subsection{Methods and Results}
Let $\omega$ be a skew-symmetric or symmetric nondegenerate
bilinear form on $\C^N$, where $N \geq 2$.  Fix a
basis $\+{e}_1, \ldots, \+{e}_N$ for $\C^N$ that is \emph{isotropic} in the sense that
\begin{equation*}\label{E:isotropic}
\omega(\+{e}_i,\+{e}_j) = \delta_{i+j,N+1} \text{ for } 1 \leq i \leq j \leq N. 
\end{equation*}
Note that if $\omega$ is symmetric, then $\omega(\+{e}_i,\+{e}_j) 
= \delta_{i+j,N+1}$ for all $i$ and $j$ in the integer interval $[1,N]$.
If $\omega$ is skew-symmetric (which can only happen when $N$ is even),
then $\omega(\+{e}_i,\+{e}_j) = -\delta_{i+j,N+1}$ for $i > j$.

For any subspace $\Sigma \subset \C^N$, we define
$\Sigma^{\perp} := \{\+w \in \C^N: \omega(\+v,\+w) = 0 \hphantom{a} \forall \+v \in \Sigma\}$.
We say $\Sigma$ is \emph{isotropic} if $\Sigma \subset \Sigma^\perp$.
Given a positive integer $m \leq \frac{N}{2}$,
the isotropic Grassmannian $IG_{\omega}(m,\C^N)$ is defined as
\[IG_{\omega}(m,\C^N) := \{\Sigma \in Gr(m,\C^N): \Sigma \subset \Sigma^\perp\}.\]
This projective variety parametrizes isotropic $m$-planes in $\C^N$.
It is said to have Lie type $C$ when $\omega$ is skew-symmetric (in which case $N$ is even),
Lie type $B$ when $\omega$ is symmetric and $N$ is odd, and Lie type $D$
when $\omega$ is symmetric and $N$ is even.
\comment{
It is said to have Lie type $C$, $B$, or $D$, depending
on whether $\omega$ is skew-symmetric or symmetric, and (in the latter case)
on whether $N$ is odd or even.}

In order to define Schubert varieties in $X := IG_{\omega}(m,\C^N)$, 
we must fix some flags on $\C^N$.
We define the \emph{standard flag} $E_{\bull}$ on $\C^N$ by
$E_j := \langle \+{e}_1, \ldots, \+{e}_j \rangle$,
the span of the first $j$ basis vectors.
In types $B$ and $C$, we define the opposite flag $E^{\op}_{\bull}$ by 
$E^{\op}_j := \langle \+{e}_{N+1-j}, \ldots, \+{e}_N \rangle$,
the span of the last $j$ basis vectors.  A more complicated
type $D$ definition is given in Section \ref{S:prelim_D}.

Given $\Sigma \in X$, the \emph{Schubert symbol of $\Sigma$} relative to $E_{\bull}$,
\[\sym(\Sigma) := \{c \in [1,N]: \Sigma \cap E_c \supsetneq \Sigma \cap E_{c-1}\},\]
records the steps $c$ in $E_{\bull}$ at which
the intersection $\Sigma \cap E_c$ jumps dimension.
\comment{
If $\+M$ is an $m \times N$ matrix with rowspace $\Sigma$, then $\sym(\Sigma)$ is the set of
pivot columns in the reduced row echelon form of $\+M$.}
Note that the set $\sym(\Sigma)$ has cardinality $m$, and
that if $c \in \sym(\Sigma)$ then $N+1-c \not\in \sym(\Sigma)$,
since $\Sigma$ is isotropic.
In general, a subset $P \subset [1,N]$ of cardinality $m$
is a \emph{Schubert symbol}\footnote{Schubert symbols are sometimes
known as \emph{jump sequences}, 
and in \cite{buch:quantum_pieri} they are referred to as \emph{index sets}.} 
if $c + d \neq N+1$ for any $c,d \in P$.
We let $\Omega(X)$ denote the set of all Schubert symbols for $X$.

\comment{
Schubert varieties in $X$ can be indexed by Schubert symbols.}
Given a Schubert symbol $P$, we define the \emph{Schubert variety} $X_P := X_P(E_{\bull})$
to be the closure in $X$ of the \emph{Schubert cell} $X^{\circ}_P(E_{\bull}) := \{\Sigma \in X: \sym(\Sigma) = P\}$.
We say $X_P$ is a Schubert variety relative to the flag $E_{\bull}$.
We also define the \emph{opposite Schubert variety} $X^P$ to be
the unique Schubert variety relative to the opposite flag $E^{\op}_{\bull}$
that intersects $X_P$ at a single point.
For the special Schubert varieties, we
adopt an additional indexing convention,
writing $X_{(r)}$ to denote
the special Schubert variety of codimension $r$ in $X$.
Given Schubert symbols $P$ and $T$, we write $T \preceq P$
if $X_T \subset X_P$.  The resulting partial order on the
set of Schubert symbols, known as the \emph{Bruhat order},
is described combinatorially in Sections \ref{S:prelim} and \ref{S:prelim_D}.

The tranvserse intersection of two Schubert varieties is known as a \emph{Richardson variety}.
Associated to Schubert symbols $P$ and $T$ we have a Richardson variety $Y_{P,T} := X_P \cap X^T$,
which is nonempty if and only if $T \preceq P$.
Since $[\cO_{X_P}] \cdot [\cO_{X^T}] = [\cO_{Y_{P,T}}]$ (see e.g. \cite{brion:lectures}),
the $K$-theoretic Pieri-type triple intersection numbers can be written
\begin{equation}\label{E:intro_triple_int_number}
 \chi([\cO_{Y_{P,T}}] \cdot [\cO_{X_{(r)}}]),
\end{equation}
where $\chi: K(X) \to \Z$ is the sheaf euler characteristic map.
These numbers are nonzero only when $T \preceq P$.

We can reinterpret this triple intersection number by means of the following incidence relation,
which consists of the two-step isotropic flag variety $IF_{\omega}(1,m,\C^N)$
and the natural projections $\psi$ and $\pi$.
In particular, we make use of the
projected Richardson variety $\psi(\pi^{-1}(Y_{P,T})) \subset \P^{N-1}$,
which is the (projectivization of the) union of all $m$-planes in the Richardson variety $Y_{P,T}$.

\begin{center}
\begin{tikzpicture}[node distance = 1.8cm]
  \node (OF) {$IF_{\omega}(1,m,\C^N)$};
  \node(dummy) [right of=OF]{};
  \node (OG) [below of=OF] {$IG_{\omega}(m,\C^N)$};
  \node (T) [right of=dummy] {$IG_{\omega}(1,\C^N) \subset \P^{N-1}$};
  \draw[->] (OF) edge node[left]{$\pi$} (OG);
  \draw[->] (OF) edge node[above]{$\psi$}(T);
  \end{tikzpicture}
\end{center}

Projected Richardson varieties like $\psi(\pi^{-1}(Y_{P,T}))$ have a number of nice geometric properties.
The recent work of He and Lam \cite{he:projected} relates these varieties
to the $K$-theory of affine Grassmannians, and it has been
proved by Billey and Coskun  \cite{billey:singularities}, and by
Knutson, Lam, and Speyer  \cite{knutson:projections}, that they
are Cohen-Macaulay with rational singularities, and that the projection map
is cohomologically trivial, in the sense that
$\psi_*[\cO_{\pi^{-1}(Y_{P,T})}] = [\cO_{\psi(\pi^{-1}(Y_{P,T}))}]$.
By this last fact, along with the projection formula,
the calculation of the triple intersection number \eqref{E:intro_triple_int_number}
amounts to showing that $\psi(\pi^{-1}(Y_{P,T}))$ 
is a complete intersection in $\P^{N-1}$
and determining the equations that define it.

A description of the projected Richardson variety $\psi(\pi^{-1}(Y_{P,T}))$
is carried out in \cite{buch:quantum_pieri},
but in the special
case that the Schubert symbols $P$ and $T$ satisfy a relation $P \to T$.
Roughly speaking, this relation signifies that $T$ shows up in some
cohomological Pieri product involving $P$.
The relation $P \to T$ requires that
$T \preceq P$, and $T \preceq P$ is a more general condition.
We note that
for $P \not\to T$,
the $K$-theoretic triple intersection numbers
$\chi([\cO_{Y_{P,T}}] \cdot [\cO_{X_{(r)}}])$ need not vanish
(in contrast to the cohomological triple intersection
numbers), and are therefore essential ingredients for the
$K$-theoretic Pieri coefficients.

When $X$ is a Grassmannian of Lie type $B$ or $C$, the authors of \cite{buch:quantum_pieri}
define a complete intersection $Z_{P,T} \subset \P^{N-1}$
for Schubert symbols $T \preceq P$ and prove that
the projected Richardson variety $\psi(\pi^{-1}(Y_{P,T}))$
is contained in it.
The authors attempt to extend the definition of $Z_{P,T}$ to the type $D$ Grassmannian 
but use an erroneous definition of Schubert varieties,
resulting in a definition of $Z_{P,T}$ that only makes sense in the special case that $P \to T$.

The first result of this paper, presented in Section \ref{S:type_D_equations}, is
to provide a corrected definition of $Z_{P,T}$ in the general setting that $T \preceq P$,
and to show that it is a complete intersection of linear and quadratic hypersurfaces. 
This process involves new combinatorics of Schubert symbols,
such as the notion of an \emph{exceptional cut}.

The second result, presented in Section \ref{S:type_D_into}, is that
$\psi(\pi^{-1}(Y_{P,T})) \subset Z_{P,T}$ for any Schubert symbols $T \preceq P$ in
a Grassmannian of Lie type $B$, $C$, or $D$.
This result generalizes \cite[Lemma 5.1]{buch:quantum_pieri}, in which the authors
prove this statement in types $B$ and $C$ only.

The third result, presented
in Section \ref{S:type_BC_onto}, is that
given a type $B$, $C$, or $D$ Grassmannian and
arbitrary Schubert symbols $T \preceq P$,
we have $Z_{P,T} \subset \psi(\pi^{-1}(Y_{P,T}))$.
We prove this result by constructing
a smaller Richardson variety contained in $Y_{P,T}$
that projects surjectively onto $Z_{P,T}$.
\comment{
More precisely, we construct
a Schubert symbol $P'$ such that
$T \preceq P' \preceq P$,
$P' \to T$,
and $Z_{P,T} = Z_{P',T}$.
By a result of \cite{buch:quantum_pieri}, since $P' \to T$, 
we then have $\psi(\pi^{-1}(Y_{P',T})) = Z_{P',T} = Z_{P,T}$.}

Combining these results, we arrive at the main theorem of this paper:
\begin{thm}\label{T:intro_main_theorem}
Let $X$ be a Grassmannian of Lie type $B$, $C$, or $D$.  For any Schubert
symbols $T \preceq P$
we have $Z_{P,T} = \psi(\pi^{-1}(Y_{P,T}))$.
\end{thm}
By Theorem \ref{T:intro_main_theorem}, we
know exactly which equations define the projected Richardson variety in all three Lie types. 
These equations allow for a pleasant calculation of the triple intersection
numbers, which we carry out in Section \ref{S:triples}.
In Section \ref{S:pieris} we describe how
$K$-theoretic Pieri coefficients are calculated as alternating sums of these triple intersection numbers.
Taken together, the results of this paper complete the story of Pieri-type triple intersection numbers
for Grassmannians.
We hope this approach will soon lead to a positive Pieri formula.

\subsection{Acknowledgments}
The results of this paper are part of the author's dissertation \cite{ravikumar:thesis}. 
The author wishes to sincerely thank 
Anders Buch for his encouragement and motivation over the years.
\comment{In addition, the author wishes to thank Chris Woodward and Leonardo Mihalcea
for their careful reading of an earlier draft of this paper.}

\section{Preliminaries 1: Types B and C}\label{S:prelim}

\subsection{Schubert Symbols}
Let $X:= IG_\omega(m,\C^N)$ be a Grassmannian of type $C$ or $B$,
where $N := 2n$ or $N := 2n+1$, depending
on whether $X$ is of type $C$ or $B$ respectively.
In the former case, we will also denote $X$
by $SG(m,2n)$ and refer to it as a \emph{symplectic Grassmannian}.
In the latter case, we will also denote $X$
by $OG(m,2n+1)$ and refer to it as an \emph{odd orthogonal Grassmannian}.
Recall that for Schubert symbols $T$ and
$P$ in $\Omega(X)$, the relation $T \preceq P$ signifies that $X_T \subset X_P$.
This partial order on the set of Schubert symbols
has a simple combinatorial description.

Given Schubert symbols $T = \{t_1 < \ldots < t_m\}$ and
$P = \{p_1 < \ldots < p_m\}$, we write
$T \leq P$ whenever $t_i \leq p_i$ for $1 \leq i \leq m$.
By \cite[Proposition 4.1]{buch:quantum_pieri} we have the following lemma:

\begin{lemma}
Provided $X$ is of Lie type $B$ or $C$, we have
$T \leq P$ if and only if $T \preceq P$.
\end{lemma}

\comment{
\begin{proof}
Given $T < P$, let $j$ be the minimum integer such that $t_j < p_j$.
Consider the map $\P^1 \to X$ given by 
\[[s:t] \mapsto \langle  \rangle\]

Thus $X^{\circ}_T \subset X_P$ for every $T \leq P$.
Now, the union of  $X^{\circ}_T \subset X_P$  is
equal to {\color{red}RANK CONDITION}, which is a closed subset of $X$
containing $X^{\circ}_P$.  Therefore,
it is equal to $X_P$, and hence
if $T \not\leq P$, we have
$T \not\preceq P$.
{\color{red} FINISH PROOF! MAY NEED TO LOOK AT BKT2 FOR TYPE B, IN ADDITION TO BKT1 FOR TYPE C.}
\end{proof}
}

For any Schubert symbol $P \in \Omega(X)$, let
$\bar{P} = \{c \in [1,N]: N+1-c \in P\}$ and let
$[P] = P \cup \bar{P}$.  Also let
$|P|$ denote the codimension
of the Schubert variety $X_P$ in $X$.

\comment{
Given a Schubert symbol $P$ such that $d = |P|$,
let $[X_P] \in H^{2d}(X)$
denote the Poincare dual of the homology class in $H_{2d}(X)$
corresponding to $X_P$.  The class $[X_P]$ is independent
of the flag $E_{\bull}$ and we refer to it as the
\emph{Schubert class corresponding to $P$}.
}

\comment{
To each Schubert symbol $P$, 
there is a unique dual symbol
$P^{\vee}$ with the property that for any Schubert symbol $T$,
\[
\rho_{*}([X_{P}] \cdot [X_{T}]) = \delta_{T,P^{\vee}},
\]
where
$\rho$ is the map from $X$ to a single point.
Equivalently,  $X_P \cap X^T$
is equal to a single point if and only if $T = P$, where $X^T := X_{T^{\vee}}(E^{\op}_{\bull})$.
The following lemma, from \cite[Proposition 4.2]{buch:quantum_pieri}, 
gives a simple description the dual symbol $P^{\vee}$.
}

For each Schubert symbol $P$, 
there is a unique dual symbol
$P^{\vee}$ with the property that for any Schubert symbol $T$,
$X_P(E_{\bull}) \cap X_{T}(E^{\op}_{\bull})$
is equal to a single point if and only if $T = P^{\vee}$.
The opposite Schubert symbol $X^P$ defined in the introduction
is therefore equal to $X_{P^{\vee}}(E^{\op}_{\bull})$.
The following lemma, from \cite[Proposition 4.2]{buch:quantum_pieri}, 
gives a simple description the dual symbol $P^{\vee}$.

\begin{lemma}\label{L:poincare_dual}
 Provided $X$ is of Lie type $B$ or $C$, we have
$P^{\vee} = \bar{P}$ for all Schubert symbols $P \in \Omega(X)$.
\end{lemma}

\comment{
\begin{proof}
 {\color{red} PROOF TO COME}
\end{proof}
}

\subsection{Richardson Diagrams}
It is a well-known fact (following from Borel's fixed-point theorem \cite{borel:fixed_point}) 
that  $T \preceq P$ if and only if $X_P \cap X^T$ is nonempty.
This variety $Y_{P,T} := X_P \cap X^T$ was shown to be reduced and irreducible in \cite{richardson:intersections},
and is known as a \emph{Richardson variety}.

Given Schubert symbols $T \leq P$,
we define the \emph{Richardson diagram} $D(P,T) = \{(j,c): t_j \leq c \leq p_j\}$, which we
represent as an $m \times N$ matrix with a $*$ for every entry in $D(P,T)$ and zeros elsewhere.
We say a matrix $(a_{i,j})$ has \emph{shape} $D(P,T)$ if its dimensions are $m \times N$ and
$a_{j,c} = 0$ for all $(j,c) \not\in D(P,T)$.
Given a matrix of shape $D(P,T)$ whose row vectors are independent and orthogonal, its rowspace
will be an element of $Y_{P,T}$.
\begin{example}
Any rank $m$ matrix of shape $D(P,P)$ will have rowspace
$\Sigma_P := \langle \+e_{p_1}, \ldots, \+e_{p_m}\rangle$,
which is the only element of $Y_{P,P}$.
\end{example}

\begin{example}\label{E:first_richardson_diagram}
Suppose $P = \{2,3,4,10\}$
and $T = \{1,2,4,6\}$
in $SG(4,10)$.  Suppose $(a_{i,j})$
is a rank $m$ matrix of shape $D(P,T)$.
The rowspace of $(a_{i,j})$ will be in $Y_{P,T}$ if and only if 
$a_{1,1}a_{4,10} + a_{1,2}a_{4,9} = 0$,
$a_{2,2}a_{4,9} + a_{2,3}a_{4,8} = 0$, and
$a_{4,7} = 0$.
We leave it to the reader to write down an explicit entries satisfying
these equations. The diagram $D(P,T)$
is shown below:
 \[
\left(
  \begin{array}{cccccccccc}
    * & * & 0 & 0 & 0 & 0 & 0 & 0 & 0 & 0  \\
    0 & * & * & 0 & 0 & 0 & 0 & 0 & 0 & 0  \\
    0 & 0 & 0 & * & 0 & 0 & 0 & 0 & 0 & 0  \\
    0 & 0 & 0 & 0 & 0 & * & * & * & * & *  \\
  \end{array}
\right).
\]
\end{example}

Given a Schubert symbol $P=\{p_1,\ldots,p_m\}$,
let $p_0 = 0$ and $p_{m+1} = N+1$.  We won't consider these
as actual elements in the Schubert symbol $P$, but the notation will
be useful.
Define a \emph{visible cut} through $D(P,T)$ to be any integer
$c \in [0,N]$ such that no row of $D(P,T)$ contains stars in both
column $c$ and column $c+1$; i.e. such that $p_i \leq c < t_{i+1}$ for some $i$.
We will consider $c=0$ and $c=N$ 
to be visible cuts.
Define an \emph{apparent cut} to be any integer $c \in [0,N]$ such that
$c$ or $N-c$ is a visible cut.  In types $B$ and $C$ we define
a \emph{cut} in $D(P,T)$ to be synonymous with an \emph{apparent cut}.
\comment{(in type $D$ we will introduce an additional type of cut).}
Let $\mathcal{C}_{P,T}$ be the set of cuts in $D(P,T)$. 

An integer $c$ is a \emph{zero column} of $D(P,T)$
if  $p_j < c < t_{j+1}$ for some $j$, since
in this case column $c$ of $D(P,T)$ has no stars.
An entry $(j,c)$ in $D(P,T)$
is a \emph{lone star} if either
\begin{enumerate}[label=\emph{\roman*})]
 \item $c \in T$ and $c$ is a cut in $D(P,T)$, or
 \item $c \in P$ and $c-1$ is a cut in $D(P,T)$. 
\end{enumerate}
The simplest example of a lone star occurs when
$t_j = p_j=c$ for some $j$.  In this case 
row $j$ and column $c$ of $D(P,T)$ each contain a
single star at $(j,c)$.
We define the set $\mathcal{L}_{P,T} \subset [1,N]$
to be the set of integers $c$ such that either
\begin{enumerate}[label=\emph{\roman*})]
 \item $c$ is a zero column in $D(P,T)$, or
 \item there exists a lone star in column $N+1-c$.
\end{enumerate}
Finally, we define the set 
\begin{equation*}
\mathcal{Q}_{P,T} :=
\begin{cases}
[0,n] \cap \mathcal{C}_{P,T} & \text{if $X$ is of type $C$,}\\
([0,n] \cap \mathcal{C}_{P,T}) \cup \{n+1\} & \text{if $X$ is of type $B$.}\\
\end{cases}
\end{equation*}

\begin{example}\label{E:rich_diagram_properties}
Continuing with Example \ref{E:first_richardson_diagram},
the set of cuts $\mathcal{C}_{P,T}$ is equal to $\{0,3,4,5,6,7,10\}$.
Of these, $0$, $3$, $4$, $5$, and $10$ are visible cuts.
Furthermore, $5$ is a zero column, $(3,4)$ is a lone star,
$\mathcal{L}_{P,T} = \{5,7\}$, and 
$\mathcal{Q}_{P,T} = \{0,3,4,5\}$.
\comment{
\[
\left(
  \begin{array}{cccccccccc}
    * & * & 0 & 0 & 0 & 0 & 0 & 0 & 0 & 0  \\
    0 & * & * & 0 & 0 & 0 & 0 & 0 & 0 & 0  \\
    0 & 0 & 0 & * & 0 & 0 & 0 & 0 & 0 & 0  \\
    0 & 0 & 0 & 0 & 0 & * & * & * & * & *  \\
  \end{array}
\right).
\]}
\end{example}

In types $B$ and $C$, lone stars take a particularly
simple form.  Namely,

\begin{prop}\label{P:simple_lone_star}
Let $X$ be an isotropic Grassmannian of Lie type $B$ or $C$.
Suppose $(j,c)$ is a lone star in $D(P,T)$.
If $c = t_j$ and $t_j$ is an apparent cut,
or if $c = p_j$ and $p_j-1$ is an apparent cut,
then either $N+1-c$ is a zero column,
or $t_j=p_j=c$.  
\end{prop}

\begin{proof}
Suppose $c=t_j$ is an apparent
cut in $D(P,T)$, and that
$N+1-c$ is not a zero column.
Since $N+1-c$ is not in $T$ \emph{and}
not a zero-column,
it follows that $N-c$ is not a visible cut.
But then $c$ must be a visible cut,
so $c = p_j$.  A similar argument
holds if we start by assuming $c=p_j$.
\end{proof}

Since zero columns are flanked by cuts,
we have the following immediate corollary.

\begin{cor}\label{C:lin_cuts}
 In types $B$ and $C$, if $c \in \mathcal{L}_{P,T}$, then $c$ and $c-1$
are both cuts.
\end{cor}

\subsection{The Projected Richardson Variety}
We now define a subvariety
of $\P^{N-1}$ that will
play a key role in the calculation of triple intersection numbers.
\comment{turn out to 
equal the projected Richardson variety $\psi(\pi^{-1}(Y_{P,T}))$.}
Let $x_1, \ldots, x_N \in (\C^N)^*$ be the dual basis to
the isotropic basis $\+e_1, \ldots, \+e_N \in \C^N$.
Let $f_0 = 0$, and for $1 \leq c \leq n$, let 
$f_c = x_1x_N + \ldots + x_cx_{N+1-c}$.
For example, $f_1 = x_1x_N$, and $f_2 = x_1x_N + x_2x_{N-1}$.
In addition, if $X$ is type $B$, 
let $f_{n+1} = x_1x_{2n+1} + \ldots + x_nx_{n+2} + \frac{1}{2}x^2_{n+1}$.
Given Schubert symbols $T \preceq P$, let
$Z_{P,T} \subset \P^{N-1}$ denote the subvariety
defined by the vanishing of the polynomials
$\{f_c \mid c \in \mathcal{Q}_{P,T}\} \cup \{x_c \mid c \in \mathcal{L}_{P,T}\}$.
We note that in the type $B$ case,
$Z_{P,T}$ must satisfy the equation $f_{n+1} = 0$
and hence lie in $OG(1,2n+1)$, the quadric
hypersurface of isotropic lines in $\P^{2n}$.

\comment{
In the case of the orthogonal Grassmannians $OG(m,2n+1)$ and $OG(m,2n+2)$,  
note that isotropic vectors must satisfy a nontrivial defining equation.
For $OG(m,2n+1)$,
it is the equation
$f^B_{n+1}
=x_1x_{N}+ \ldots + x_nx_{n+2} + \frac{1}{2}x^2_{n+1}=0$.
Motivated by this inherent quadratic equation, we let
$\mathcal{Q}_{P,T} = ([0,n] \cap \mathcal{C}) \cup \{n+1\}$
for any $T \preceq P$.
}

In fact, $Z_{P,T}$ is a complete intersection in $\P^{N-1}$ cut out by the polynomials:
\begin{enumerate}[label=\emph{\alph*})]
 \item $f_d - f_c = x_{c+1}x_{N-c} + \ldots + x_dx_{N+1-d}$ 
 if $c$ and $d$ are consecutive elements of $\mathcal{Q}_{P,T}$ such that $d-c \geq 2$, and
 \item $x_c$ if $c \in \mathcal{L}_{P,T}$.
\end{enumerate}
We will prove this fact for all three Lie types in Proposition \ref{P:complete_intersection_BCD}.

Recall that we have projections $\pi$ and $\psi$ from the flag variety
$IF_{\omega}(1,m,\C^N)$ to $X$ and $IG_{\omega}(1,\C^N)$ respectively.
The variety $\pi^{-1}(Y_{P,T})$ is a Richardson variety in $IG_{\omega}(1,m,\C^N)$,
and its image $\psi(\pi^{-1}(Y_{P,T}))$ is known as a \emph{projected Richardson variety}.
We shall prove that the projected Richardson variety $\psi(\pi^{-1}(Y_{P,T}))$
is in fact equal to $Z_{P,T}$.  One direction is straightforward.

\begin{lemma}\label{L:type_B_C_into}
Given Schubert symbols $T \preceq P$ for a Grassmannian $X$ of 
Lie type $B$ or $C$, we have  $\psi(\pi^{-1}(Y_{P,T})) \subset Z_{P,T}$.
\end{lemma}

A proof of Lemma \ref{L:type_B_C_into} can be found in 
\cite[Lemma 5.1]{buch:quantum_pieri}.
This proof is correct for types $B$ and $C$, but
does not go through in type $D$ due to an erroneous definition of the Bruhat order.
We supply a corrected proof for all three Lie types in Section \ref{S:type_D_into}.

\begin{example}
Continuing with Example \ref{E:rich_diagram_properties},
suppose $\+M$ is a matrix with shape $D(P,T)$
and independent, isotropic row vectors.
Note that any vector in the rowspace of
$\+M$ must satisfy the quadratic equation 
$x_1x_{10} + x_2x_9 + x_3x_8=0$
and the linear equations
$x_5=0$ and $x_7 = 0$,
which are precisely the equations
defining $Z_{P,T}$.
By Lemma \ref{L:type_B_C_into},
\emph{any} vector contained in an $m$-plane $\Sigma \in Y_{P,T}$
satisfies these equations.
\end{example}

\section{Preliminaries 2: Type D}\label{S:prelim_D}
Consider $\C^{2n+2}$ endowed with a nondegenerate symmetric bilinear form.
Let $X := OG(m,2n+2)$ denote the \emph{even orthogonal Grassmannian}
of isotropic $m$-planes in $\C^{2n+2}$.
For any Schubert symbol $P \in \Omega(X)$, let
$\bar{P} = \{c \in [1,2n+2]: 2n+3-c \in P\}$ and let
$[P] = P \cup \bar{P}$. As before, let
$|P|$ denote the codimension
of the Schubert variety $X_P$ in $X$.
We define 
$\type(P) \in \{0,1,2\}$ as follows.
If $n+1 \in [P]$,
then we let $\type(P)$
be congruent mod $2$ to the
number of
elements in $[1,n+1] \setminus P$.  In other words,
if $\#([1,n+1] \setminus P)$ is even
then $\type(P)=0$, and
if $\#([1,n+1] \setminus P)$ is odd
then $\type(P)=1$.
Finally, if $\{n+1, n+2\} \cap P = \emptyset$,
we set $\type(P)=2$.
\footnote{$\type(P)$ differs slightly from the \cite{buch:even_orthogonal} function type$(P)$.
Namely, $\text{type}(P) \equiv \type(P)+1 \pmod{3}$.}
\comment{
For $\Sigma \in X$, we define
$\type(\Sigma)$ to be $\type(\sym_{E_{\bull}}(\Sigma))$.
We note that
whenever
$\Sigma \cap E_{n+2} \neq \Sigma \cap E_{n}$,
$\type(\sym(\Sigma))$ is congruent mod $2$ to
the codimension of $\Sigma \cap E_{n+1}$ in $E_{n+1}$.
We will always measure $\type(\Sigma)$ with
respect to the standard flag $E_{\bull}$.
}

The following proposition is due to
\cite[Proposition A.2]{buch:even_orthogonal}.

\begin{prop}\label{P:type_D_bruhat_order}
Given Schubert symbols $P$ and $T$ in $\Omega(OG(m,2n+2))$,
we have $T \preceq P$ if and only if
\begin{enumerate}[label=\emph{\roman*})]
\item $T \leq P$, and
\item if there exists $c \in [1,n]$ such that $[c+1,n+1] \subset [P] \cap [T]$
and $\#P \cap [1,c] = \#T \cap [1,c]$, then
we have $\type(P) = \type(T)$.
\end{enumerate}
\end{prop}

By Proposition \ref{P:type_D_bruhat_order}, the type $D$ Bruhat order is not 
simply the $\leq$ ordering.
The following example illustrates the difference.

\begin{example}\label{E:OG(2,6)}
 The $\preceq$ partial order is shown below for
 the Schubert symbols on $OG(2,\C^6)$, which are colored by type.  
 Notice that there are six ``missing'' edges, which
 would have occurred had we used the (incorrect) $\leq$ ordering.
 \begin{center}
 \begin{tikzpicture}
  \node (min) at (0,0) {$\{1,2\}$};
  \node (a) at (2,1) {$ {\color{red}\{1,4\}}$};
  \node (b) at (2,-1) {$ {\color{blue}\{1,3\}}$};
  \node (c) at (4,2) {$ {\color{red}\{2,4\}}$};
  \node (d) at (4,0) {$\{1,5\}$};
  \node (e) at (4,-2) {$ {\color{blue}\{2,3\}}$};
  \node (f) at (6,2) {$ {\color{red}\{3,5\}}$};
  \node (g) at (6,0) {$\{2,6\}$};
  \node (h) at (6,-2) {$ {\color{blue}\{4,5\}}$};
  \node (i) at (8,1) {$ {\color{red}\{3,6\}}$};
  \node (j) at (8,-1) {$ {\color{blue}\{4,6\}}$};
  \node (max) at (10,0) {$\{5,6\}$};
  \draw (min) -- (a) -- (c) -- (f) -- (i) -- (max)
(min) -- (a) -- (c) -- (f) -- (i) -- (max)
(a) -- (d) -- (g) -- (i)
(b) -- (d)
(g) -- (j)
(c) -- (g)
(d) -- (f)
(d) -- (h)
(e) -- (g)
  (min) -- (b) -- (e) -- (h) -- (j) -- (max);
  \draw[preaction={draw=white, -,line width=6pt}];
\end{tikzpicture}
\end{center} 
\end{example}

We define the opposite flag $E^{\op}_{\bull}$ by
$E^{\op}_j = \langle \+e_{2n+3-j}, \ldots, \+e_{2n+2} \rangle$ for $j \neq n+1$,
and
\begin{equation*}
E^{\op}_{n+1} =
\begin{cases}
\langle \+e_{n+2}, \+e_{n+3} ,\ldots, \+e_{2n+2}  \rangle & \text{ if } n \text{ is odd}, \\
\langle \+e_{n+1}, \+e_{n+3} ,\ldots, \+e_{2n+2}  \rangle & \text{ if } n \text{ is even}. \\
\end{cases}
\end{equation*}
This definition guarantees that $E_{\bull}$ 
and $E^{\op}_{\bull}$ lie in the same connected component of
the variety of complete isotropic flags on $\C^{2n+2}$ 
(endowed with a nondegenerate symmetric bilinear form),
which is disconnected.  
\comment{
As a consequence,
$[\cO_{X_P(E^{\op}_{\bull})}] = [\cO_{X_P(E_{\bull})}]$ for any Schubert symbol $P$.
}

Let $\iota$ be the permutation of $\{1, \ldots, 2n+2\}$ that
interchanges $n+1$ and $n+2$ and leaves all other numbers fixed.
Given a type $D$ Schubert symbol $P = \{p_1, \ldots, p_m\}$, 
let $\iota(P)=\{\iota(p_1), \ldots, \iota(p_m)\}$. 
From \cite[p. 43]{buch:quantum_pieri},
we have the following description of the dual symbol $P^{\vee}$:

\begin{lemma}
Given a Schubert symbol $P \in \Omega(OG(m,2n+2))$, we have
\begin{equation*}
P^{\vee} =
\begin{cases}
 \bar{P} & \text{ when $n$ is odd}, \\
 \iota(\bar{P}) & \text{ when $n$ is even}. \\
\end{cases}
\end{equation*}
\end{lemma}

\comment{
\begin{proof}
{\color{red} PROOF TO COME!} 
\end{proof}
}

If the type $D$ definitions of opposite flags and dual Schubert symbols
appear confusing, the following observation may offer some relief:

\begin{obs}\label{O:two_wrongs}
Let $X := IG_{\omega}(m,N)$ be a Grassmannian of type $B$, $C$, or $D$.
For any Schubert symbol $P$, we have
\[
X^{\circ}_{P^{\vee}}(E^{\op}_{\bull}) = \{\Sigma \in IG:
\Sigma \cap \langle \+e_{p_i}, \ldots, \+e_{N} \rangle \supsetneq 
\Sigma \cap \langle \+e_{p_i+1}, \ldots, \+e_{N} \rangle \}.
\]
\end{obs}

Observation \ref{O:two_wrongs} is obvious unless we are working in $OG(m,2n+2)$
and $n$ is even.  We illustrate that case in the following example.

\begin{example}
Consider $OG(1,6)$, and let $P = \{4\}$.
Then $P^{\vee} = \{4\}$, 
$E^{\op}_{3} = \langle \+e_3, \+e_5, \+e_6 \rangle$, and
$E^{\op}_{4} = \langle \+e_3, \+e_4, \+e_5, \+e_6 \rangle$.
By definition, 
$X^{\circ}_{P^{\vee}}(E^{\op}_{\bull}) = \{\Sigma \in IG:
\Sigma \cap E^{\op}_4 \supsetneq \Sigma \cap E^{\op}_3\}$,
which is equal to the set of points in $\P^5$ of the form
$\langle (0, 0, 0, 1, *, *) \rangle$,
in agreement with Observation \ref{O:two_wrongs}.
\end{example}

By Observation \ref{O:two_wrongs},
any element of the Schubert cell $X^{\circ}_{P^{\vee}}(E^{\op}_{\bull})$
is the rowspace of an isotropic $m \times N$ matrix $(a_{i,j})$
with $a_{i,p_i}=1$ for $1 \leq i \leq m$ and
$a_{i,j} = 0$ for $j < p_i$.

\begin{example}
\comment{
 Consider $OG(3,8)$, and let $P = \{2,3,5\}$.  
 Then $P^{\vee} = \{4,6,7\}$, and 
 $X^{\circ}_{P^{\vee}}(E^{\op}_{\bull})$ consists
 of elements of $OG(3,8)$
 that are equal to the span of the rows of
 a matrix of the form
 \[
 \left(
  \begin{array}{cccccccc}
    0 & 1 & * & * & * & * & * & * \\
    0 & 0 & 1 & * & * & * & * & * \\
    0 & 0 & 0 & 0 & 1 & * & * & * \\
  \end{array}
\right).
\]
}
Consider $OG(3,10)$, and let  $P = \{1,4,5\}$.
In this case, $P^{\vee} = \{5,7,10\}$.
We can
write any element of  $X^{\circ}_{P^{\vee}}(E^{\op}_{\bull})$
as the rowspace of a matrix of the form 
 \[
 \left(
  \begin{array}{cccccccccc}
    1 & * & * & * & * & * & * & * & * & * \\
    0 & 0 & 0 & 1 & * & * & * & * & * & * \\
    0 & 0 & 0 & 0 & 1 & 0 & * & * & * & * \\
  \end{array}
\right).
\]
\end{example}

Thus $X_{T^{\vee}}(E^{\op}_{\bull}) = X^T$ for any Schubert symbol $T$,
since $X_T \cap X_{T^{\vee}}(E^{\op}_{\bull})$ is a single point.
We define the Richardson variety $Y_{P,T} := X_P \cap X^{T}$.  As before,
$Y_{P,T} \neq \emptyset$ if and only if $T \preceq P$.
We define the Richardson diagram $D(P,T) := \{(j,c): t_j \leq c \leq p_j\}$ 
for any Schubert symbols $T \leq P$.  This definition holds
when $T \not\preceq P$, but in this case there cannot exist
a matrix of shape $D(P,T)$ whose row vectors are independent and orthogonal
(a fact we shall prove in Proposition \ref{P:alt_bruhat}).

\begin{example}
There are no matrices $(a_{i,j})$ of shape
$D(\{2,5,7,8\},\{1,3,4,6\})$ whose rows span
an element of $OG(4,10)$, because
the isotropic relations on the entries are inconsistent:
\[
\left(
  \begin{array}{cccccccccc}
    a_{1,1} & a_{1,2} & 0 & 0 & 0 & 0 & 0 & 0 & 0 & 0 \\
    0 & 0 & a_{2,3} & a_{2,4} & a_{2,5} & 0 & 0 & 0 & 0 & 0 \\
    0 & 0 & 0 & a_{3,4} & a_{3,5} & a_{3,6} & a_{3,7} & 0 & 0 & 0 \\
    0 & 0 & 0 & 0 & 0 & a_{4,6} & a_{4,7} & a_{4,8} & 0 & 0 \\
  \end{array}
\right).
\]
We leave it to the reader to verify this fact,
as well as the fact that
$\{1,3,4,6\} \not\preceq \{2,5,7,8\}$
in type $D$.
\end{example}

\section{Result 1: Defining $Z_{P,T}$ in Type D}\label{S:type_D_equations}
Let $X := OG(m,2n+2)$ be a type $D$ Grassmannian, let $N:=2n+2$, and let $T \preceq P$
be Schubert symbols in $\Omega(X)$.
Visible cuts, apparent cuts, lone stars, and zero columns in $D(P,T)$
are defined exactly as in types $B$ and $C$.  Similarly, 
$\mathcal{C}_{P,T}$ continues to denote the set of all cuts in $D(P,T)$,
and $\mathcal{L}_{P,T}$ continues to denote the set of integers $c \in [1,2n+2]$
such that either $c$ is a zero column \comment{in $D(P,T)$}
or column $2n+3-c$ \comment{of $D(P,T)$} contains a lone star.
However, in order to define the subvariety $Z_{P,T} \subset \P^{2n+1}$,
we must define a new type of cut in $D(P,T)$.

\subsection{Exceptional Cuts}
If for some $i$ we have $p_i = n+2 \leq t_{i+1}$
or $t_i = n+1 \geq p_{i-1}$, we let $n+1$ be a cut in $D(P,T)$,
which we will refer to as an \emph{exceptional center cut}.
This cut will induce a lone star in column $n+2$ or $n+1$
respectively.

\begin{example}
$P = \{2, 4\}$ and $T = \{1, 2\}$
in $OG(2,6)$.  $D(P,T)$ is shown below,
and has an exceptional center cut.  As
a result, $(2,4)$ is a lone star,
and $3 \in \mathcal{L}_{P,T}$
\[
\left(
  \begin{array}{ccc|ccc}
    * & * & 0 & 0 & 0 & 0 \\
    0 & * & * & * & 0 & 0 \\
  \end{array}
\right).
\]
\end{example}

There are additional exceptional cuts in $D(P,T)$.
Let $c \in [1,n]$ be a \emph{cut candidate} if
$[c+1,n+1] \subset [P] \cap [T]$ and $\#(T\cap[1,c])$ = $\#(P\cap[1,c]) + 1$. 
If $\type(T) \neq \type(P)$, then $c$ and $N+1-c$
will also be cuts in $D(P,T)$, for each cut candidate $c$. 
We'll refer to these as \emph{exceptional cuts} as well.
We give several examples of diagrams with exceptional cuts,
as the definition is somewhat complicated.

\begin{example}
$P = \{3, 6\}$ and $T = \{2, 3\}$
in $OG(2,6)$.  $D(P,T)$ is shown below,
and $\mathcal{C} = \{0,1,2,3,4,5,6\}$.
Of these, $2$, $3$ (the center cut), and $4$ are
exceptional cuts.  $(1,2), (1,3)$, and $(2,3)$
are all lone stars, and
$\mathcal{L}_{P,T} = \{1,4,5\}$.
 \[
\left(
  \begin{array}{cc|c|c|cc}
    0 & * & * & 0 & 0 & 0 \\
    0 & 0 & * & * & * & * \\
  \end{array}
\right).
\]
\end{example}

\begin{example}
$P = \{3, 4, 7\}$ and $T = \{1,3, 4\}$
in $OG(3,8)$.  $D(P,T)$ is shown below,
and $\mathcal{C} = \{0,1,2,3,4,5,6,7,8\}$.
Of these, $2$, $3$, $4$ (the center cut), $5$,and $6$ are
exceptional cuts. By finding all the lone stars,
one can check that
$\mathcal{L}_{P,T} = \{2,5,6,8\}$.
 \[
\left(
  \begin{array}{cc|c|c|c|c|cc}
    * & * & * & 0 & 0 & 0 & 0 & 0\\
    0 & 0 & * & * & 0 & 0 & 0 & 0\\
    0 & 0 & 0 & * & * & * & * & 0\\    
  \end{array}
\right).
\]
\end{example}

\begin{example}
$P = \{4, 6, 8\}$ and $T = \{1, 3, 5\}$
in $OG(3,8)$.  $D(P,T)$ is shown below,
and $\mathcal{C} = \{0,2,6,8\}$.
Of these, $2$ and $6$ are
exceptional cuts, and
$\mathcal{L}_{P,T} = \emptyset$.
\[
\left(
  \begin{array}{cc|cccc|cc}
    * & * & * & * & 0 & 0 & 0 & 0 \\
    0 & 0 & * & * & * & * & 0 & 0 \\
    0 & 0 & 0 & 0 & * & * & * & * \\
  \end{array}
\right).
\]
\end{example}
\begin{example}
$P = \{4, 5, 8, 9\}$ and $T = \{1, 3,4, 6\}$
in $OG(4,10)$.  $D(P,T)$ is shown below,
and $\mathcal{C} = \{0,1,2,8,9,10\}$.
Of these, $2$ and $8$ are
exceptional cuts, and
$\mathcal{L}_{P,T} = \{2,10\}$.
\[
\left(
  \begin{array}{cc|cccccc|cc}
    * & * & * & * & 0 & 0 & 0 & 0 & 0 & 0\\
    0 & 0 & * & * & * & 0 & 0 & 0 & 0 & 0 \\
    0 & 0 & 0 & * & * & * & * & * & 0 & 0 \\
    0 & 0 & 0 & 0 & 0 & * & * & * & * & 0 \\
  \end{array}
\right).
\]
\end{example}

\comment{
\subsection{Additional Features of Diagram}
}

We relate certain features of the Richardson diagram $D(P,T)$
to the type $D$ Bruhat order and to the existence of exceptional cuts.

\begin{lemma}\label{L:another_condition_for_maximal_projection}
 For any Schubert symbol $P$, the following conditions are equivalent:
 \begin{enumerate}
  \item  $[c+1, n+1] \subset [P]$
  \item  $\#([c+1,N-c] \cap P) = n+1-c$.
 \end{enumerate}
\end{lemma}

\begin{proof}
 Note that $n+1-c = \#([c+1,n+1])$.  Because $P$ is an isotropic
 Schubert symbol, there can be at most $n+1-c$ elements
 in $[c+1,N-c] \cap P$.  Since $[c+1, n+1] \subset [P]$,
 there are at least that many.
\end{proof}

\begin{lemma}\label{L:mirror_crossings_same_cardinality}
 Given Schubert symbols $T \leq P$ such that $[c+1,n+1] \subset [T] \cap [P]$,
 we have
 \[
\#([1,c] \cap T) - \#([1,c] \cap P) = \#([N+1-c,N] \cap P) - \#([N+1-c,N] \cap T).  
 \]
\end{lemma}

\begin{proof}
By Lemma \ref{L:another_condition_for_maximal_projection},
$\#([c+1,N-c] \cap P) = \#([c+1,N-c] \cap T) = n+1-c$.
It follows that 
\begin{align*}
&\#([1,c] \cap P) + \#([N+1-c,N] \cap P)\\
&= m - (n+1-c)\\
&= \#([1,c] \cap T) + \#([N+1-c,N] \cap T).\\
\end{align*}
\end{proof}

Lemma \ref{L:mirror_crossings_same_cardinality} says that
whenever $[c+1,n+1] \subset [T] \cap [P]$,
the number of rows crossing from column $c$ to column $c+1$ of $D(P,T)$ is equal
to the number of rows crossing from column $N-c$ to column $N+1-c$ of $D(P,T)$.
We therefore have the following corollary.

\begin{cor}\label{C:ways_to_think_about_critical_windows}
Given $c \in [1,n]$,
 suppose $[c+1,n+1] \subset [T] \cap [P]$ and $\type(P) \neq \type(T)$ for
 Schubert symbols $T \leq P$.
 We then have the following two sets of implications.

  \begin{align*}
   \mathbf{1.} \hphantom{a} & \#([1,c] \cap T) = \#([1,c] \cap P) \\
   &\Longleftrightarrow \#([N-c+1,N] \cap T) = \#([N-c+1,N] \cap P) \\
   &\Longleftrightarrow c \text{ is a visible cut in $D(P,T)$} \\
   &\Longleftrightarrow N-c \text{ is a visible cut in $D(P,T)$} \\
   &\Longrightarrow T \not\preceq P. \\
   \\
   \mathbf{2.} \hphantom{a} & c \text{ and } N-c \text{ are exceptional cuts in $D(P,T)$} \\
   &\Longleftrightarrow\#([1,c] \cap T) = \#([1,c] \cap P)+1 \\
   &\Longleftrightarrow \#([N-c+1,N] \cap P) = \#([N-c+1,N] \cap T)+1 \\
   &\Longleftrightarrow  D(P,T) \text{ has exactly one row crossing from column } c \text{ to column } c+1 \\
   &\Longleftrightarrow  D(P,T) \text{ has exactly one row crossing from column } N-c \\
   &\hphantom{\Longleftrightarrow} \text{ to column } N-c+1 \\
  \end{align*}
\end{cor}

\comment{
\subsection{Further Insight into the Type D Bruhat Order}
}

We finish this section by proving that
several important properties of Richardson diagrams
carry over to the type $D$ case.
In particular we extend Corollary \ref{C:lin_cuts} to type $D$,
and then prove in Corollary \ref{C:no_lin_in_PT} that 
$(P \cup T) \cap \mathcal{L}_{P,T} = \emptyset$ (a fact
that is obvious in types $B$ and $C$).
Once these facts are established we will be ready to 
define $Z_{P,T}$.

First we observe that for any $T \leq P$,
$D(P,T)$ and $D(\bar{T},\bar{P})$
have the same cut candidates,
by Lemma \ref{L:mirror_crossings_same_cardinality}. It follows that:

\begin{obs}\label{O:rotate_cuts}
$180^{\circ}$ rotation of the diagram $D(P,T)$ preserves all cuts, including exceptional cuts.
In other words, $\mathcal{C}_{P,T} = \mathcal{C}_{\bar{T},\bar{P}}$.
\end{obs}

\comment{
{\color{red} A LITTLE WEIRD TO HAVE BOTH LOWER AND UPPER CASE N'S IN THE FOLLOWING PROOF}}

We can now prove the type $D$ version of Corollary \ref{C:lin_cuts}.

\begin{prop}\label{P:lin_cuts_all_types}
 In type $D$, if $c \in \mathcal{L}_{P,T}$, then $c$ and $c-1$
are both in $\mathcal{C}_{P,T}$.
\end{prop}

\begin{proof}
If $c$ is a zero column then the result is clear.
Otherwise, it must be the case that 
$(i,N+1-c)$ is a lone star for some $i$.
By Observation \ref{O:rotate_cuts},
we can assume without loss of generality that
$N+1-c  \leq n+1$.

{\bf Case 1}: $N+1-c = t_i$ and $t_i$ is a cut in $D(P,T)$.

We claim that $t_i -1$ must be a cut as well.
If $t_i = p_i$, then $p_{i-1} < t_i$, 
and we are done. 
Thus, we only need to consider the case that
$t_i$ is an exceptional cut in $D(P,T)$.

If $t_i-1$ is not a visible cut, then
$p_{i-1} \geq t_i$.
In fact, if $t_i = n+1$, then $p_{i-1} = t_i$,
since that is the only way the exceptional center cut
can arise.  On the other hand, if $t_i \neq n+1$,
then since $\#([1,t_i] \cap T) = \#([1,t_i] \cap P) + 1$,
row $i$ of $D(P,T)$ is the \emph{only} row crossing the exceptional cut $t_i$.
In this case too we must have $p_{i-1} = t_i$.
 
We therefore have 
$t_i \in [T] \cap [P]$.  Furthermore, since
row $i-1$ is the only row crossing from column
$t_i-1$ to column $t_i$, we have
$\#([1,t_i-1] \cap T) = \#([1,t_i-1] \cap P)+1$.
Thus $t_i -1$ is also an exceptional cut in $D(P,T)$.

{\bf Case 2}: $N+1-c = p_i$ and $p_i-1$ is a cut in $D(P,T)$.

We claim that $p_i$ must be a cut as well.
As before, we can assume that $p_i-1$
is an exceptional cut in $D(P,T)$.

If $p_i$ is not a visible cut, then
$t_{i+1} \leq p_i$.
In fact, we must have
$t_{i+1} = p_i$, since
row $i$ is the only row crossing the
exceptional cut $p_i-1$.

If $p_i = n+1$, then since $t_{i+1} = p_i$, 
the diagram $D(P,T)$ has the exceptional center cut $n+1$,
and we are done.

If $p_i \neq n+1$, then since $t_{i+1} = p_i$,
row $i+1$ must be the only row
crossing from column $p_i$ to column
$p_i+1$.
Hence,
$\#([1,p_i] \cap T) = \#([1,p_i] \cap P)+1$.
Thus $p_i$ is also an exceptional cut in $D(P,T)$.
\comment{
By Proposition \ref{P:simple_lone_star}, 
we only need to consider the case of a lone star $(i,c)$ 
for which $c$ or $c-1$ is an exceptional cut.
By Observation \ref{O:rotate_cuts} we can
assume $c \leq n+1$.

Suppose first that $c=t_i$ for some $i$, and that 
$t_i$ is an exceptional cut in $D(P,T)$.
If $t_i-1$ is not a visible cut, then
$p_{i-1} \geq t_i$.
However, since $\#([1,t_i] \cap T) = \#([1,t_i] \cap P) + 1$,
there is exactly one row crossing the exceptional cut $t_i$,
so we must have $p_{i-1} = t_i$.
 
We therefore have 
$t_i \in [T] \cap [P]$.  Furthermore, since
row $i-1$ is the only row crossing from
$t_i-1$ to $t_i$, we have
$\#([1,t_i-1] \cap T) = \#([1,t_i-1] \cap P)+1$.
Thus $t_i -1$ is also an exceptional cut in $D(P,T)$.

On the other hand, suppose $c=p_i$ for some $i$, and $p_i-1$
is an exceptional cut.  If $p_i$ is not a visible cut,
then by a similar argument, $t_{i+1} = p_i$.
It follows that 
$\#([1,p_i] \cap T) = \#([1,p_i] \cap P) +1$,
and $p_i$ is an exceptional cut as well.
}
\end{proof}

\comment{
We shall now give an alternative characterization of the Bruhat
order in type $D$.}  

Given Schubert symbols $T \leq P$ in $\Omega(OG(m,2n+2))$ such that $\type(T) \neq \type(P)$, 
\comment{we say the 
$i^{\text{th}}$ row of the diagram $D(P,T)$ is a \emph{center crossing}
if $t_i \leq n+1 < p_i$.}
we define a \emph{critical window} in $D(P,T)$ to be an interval $[c+1,N-c]$
such that $c$ and $N-c$ are visible cuts in $D(P,T)$, 
and $[c+1,n+1] \subset [T] \cap [P]$.
\comment{
\begin{itemize}
\item $c$ and $N-c$ are visible cuts in $D(P,T)$, 
\item $[c+1,n+1] \subset [T] \cap [P]$, and
\item there are an odd number of center crossings in $D(P,T)$.
\end{itemize}}

\begin{lemma}\label{L:critical_window_bruhat}
Given Schubert symbols $T<P$ in $\Omega(OG(m,2n+2))$, 
we have $T \not\prec P$ if and only if a critical window exists in $D(P,T)$.
\end{lemma}

\begin{proof}
If $T \not\prec P$, then $\type(P) \neq \type(T)$
and there exists $c \in [1,n]$ such that  $[c+1,n+1] \subset [T] \cap [P]$
and $\#[1,c] \cap P = \#[1,c] \cap T$.
By Corollary \ref{C:ways_to_think_about_critical_windows},
both $c$ and $N-c$ are visible cuts in $D(P,T)$, and hence
$[c+1,N-c]$ is a critical window.
Conversely, if $D(P,T)$ has a critical window, then \comment{$\type(P) \neq \type(T)$
and there exists a visible cut $c$ such that  $[c+1,n+1] \subset [T] \cap [P]$.
The fact that $c$ is a visible cut implies that  $\#[1,c] \cap P = \#[1,c] \cap T$,
so $T \not\prec P$.} it is clear that $T \not\prec P$.
\comment{
 Note that $c$ is a visible cut in $D(P,T)$ if and only if 
 $\#[1,c] \cap P = \#[1,c] \cap T$.
 
 Furthermore, if there exists an integer $c$ such
 that $[c+1, n+1] \subset [T] \cap [P]$  \emph{and}
 $\#[1,c] \cap P = \#[1,c] \cap T$,
 then $N-c$ must also be a cut in $D(P,T)$.
 Therefore, the first two conditions characterizing 
 a critical window are equivalent to the first two conditions
 for $T \not\prec P$, given $T < P$.
 
 Finally, given that the first two conditions hold,
 the third condition for a critical window (that there are an odd number of center crossings in $D(P,T)$),
 is equivalent to the third condition for $T \not\prec P$ (that $\type(P) \neq \type(T)$).}
\end{proof}

The fact that $(P \cup T) \cap \mathcal{L}_{P,T} = \emptyset$ follows easily
from the following proposition.

\begin{prop}\label{P:alt_bruhat}
 Given $T$ and $P$ in $OG(m,2n+2)$ such that $T < P$.  
 Then $T \not\prec P$ if and only there exists an integer $d \in [1,N]$ such that
 $D(P,T)$ has lone stars in columns $d$ and $N+1-d$.
\end{prop}

\begin{proof}

Suppose columns $d$ and $N+1-d$ of 
$D(P,T)$ both contain lone stars, and assume $d \leq n+1$.
If $d = t_i$ for some $i$, then $N+1-d = p_j$ for some $j$, as
shown in the left hand side of Figure \ref{F:conflicting_lone_stars}.
It follows that $t_i \neq p_i$, so $t_i$ must be an exceptional cut in $D(P,T)$.
Thus $[t_i+1,n+1] \subset [T] \cap [P]$ and $\type(T) \neq \type(P)$.
Furthermore, row $i$ must be the only row crossing from column $t_i$ to column $t_i+1$,
and hence $p_{i-1} < t_i$, implying that $t_i-1$ is a visible cut.
Therefore, $[t_i, p_j]$ is a critical window in $D(P,T)$.
On the other hand, if $ d= p_i$ for some $i$, then
$N+1-d = t_j$ for some $j$, as shown in the right hand side of \ref{F:conflicting_lone_stars}.
In this case $p_i-1$ must be an exceptional cut,
$p_i$ must be a visible cut, and
$[p_i+1,t_i-1]$
must be a critical window in $D(P,T)$.
By Lemma \ref{L:critical_window_bruhat}, it follows that
$T \not\preceq P$. 

\begin{figure}
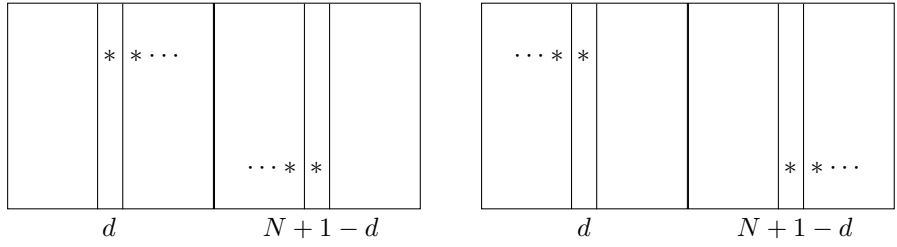 
\caption{Conflicting lone stars in $D(P,T)$ in Proposition \ref{P:alt_bruhat}.}
\label{F:conflicting_lone_stars}
\end{figure}

On the other hand, if $T \not\preceq P$, then
by Lemma \ref{L:critical_window_bruhat} there exists a critical
window $[c+1,N-c]$ in $D(P,T)$.
We claim that there exists
a (possibly smaller) critical window of the form $[t_i,p_j]$ for some $i$ and $j$.
To see why, note that if $[c+1,N-c]$ does not have the form $[t_i,p_j]$,
then either $c+1 = t_i = p_i$, or $N-c = t_j = p_j$.  Either way,
$[c+2,N-c-1]$ is a smaller critical window.
However, this process of shrinking can't continue indefinitely.
In particular, if $[c+1,N-c] = [n+1,n+2]$, then the
fact that $[n+1,n+2]$ is a critical window
implies that $t_i=n+1$ and $p_i = n+2$ for some $i$.

Finally, note that
whenever $[t_i,p_j]$ is a critical window in $D(P,T)$,
it must be the case that
$t_i$ and $p_j-1$ are exceptional cuts.
Thus $(i,t_i)$ and $(j,p_j)$ are lone stars
and $t_i + p_j = N+1$, completing the proof.

\end{proof}

\begin{cor}\label{C:no_lin_in_PT}
Given $T \preceq P$, we have
$(P \cup T) \cap \mathcal{L}_{P,T} = \emptyset$.
\end{cor}

\begin{proof}
 Suppose $c \in (P \cup T) \cap \mathcal{L}_{P,T}$. 
 Since $c$ is not a zero column, 
 there exists a lone star in column $N+1-c$ of $D(P,T)$.
 By Proposition \ref{P:lin_cuts_all_types},
 $c$ and $c-1$ are both cuts in $D(P,T)$.
 But then column $c$ contains a lone star as well,
 since $c \in (P \cup T)$,
 contradicting Proposition \ref{P:alt_bruhat}.
\end{proof}

\comment{
\begin{cor}\label{C:both_zero_columns}
 Given $T \preceq P$, if $c$ and $N+1-c$ are both in $\mathcal{L}_{P,T}$, then
 $c$ and $N+1-c$ are both zero columns.
\end{cor}

\begin{proof}
If $c$ is not a zero column, then
$N+1-c$ is a lone star. But then
$N+1-c \in P \cup T$, contradicting
Corollary \ref{C:no_lin_in_PT}.
The same reasoning
shows that $N+1-c$ is also a zero column.
\end{proof}
}

\subsection{A Complete Intersection}

The quadratic equation characterizing isotropic vectors
in $\C^{2n+2}$ is $f^D_{n+1} := x_1x_N + \ldots + x_{n+1}x_{n+2}=0$.
We once again let
$\mathcal{Q}_{P,T} = ([0,n]\cap \mathcal{C}) \cup \{n+1\}$.
We let $Z_{P,T} \subset \P^{2n+1}$ denote the subvariety cut out by 
the familiar polynomials
$\{f_c \mid c \in \mathcal{Q}_{P,T}\} \cup \{x_c \mid c \in \mathcal{L}_{P,T}\}$,
where we let $f_{n+1} = f^D_{n+1}$.

It is not immediately obvious that $Z_{P,T}$ is a complete intersection
in $\P^{N-1}$, or even that it is an irreducible subvariety.
To prove these facts, we need the following lemma, which we prove
in all three Lie types.

\begin{lemma}\label{L:consec_cuts}
Given Schubert symbols $T \preceq P$ for a Grassmannian $X$ of 
Lie type $B$, $C$, or $D$,
 if $c-1$ and $c$ are both in $\mathcal{Q}_{P,T}$,
 then $c \in \mathcal{L}_{P,T}$ or $N+1-c \in \mathcal{L}_{P,T}$.
\end{lemma}

\begin{proof}
If we are working in type $B$, and $c = n+1$,
then $n+1$ must be a zero column, and hence be
in $\mathcal{L}_{P,T}$.
Otherwise, we can assume that
$c \leq \floor{N/2}$.
If either $c$ or $N+1-c$ is a zero column in $D(P,T)$
then we are done, so assume neither column is empty.

Note that if $c-1$ is an exceptional cut,
then $c \in [T] \cap [P]$.
Otherwise, either $c-1$ or $N+1-c$ is a visible cut in $D(P,T)$,
and hence $c \in T$ or $N+1-c \in P$ respectively,
since neither $c$ nor $N+1-c$ is a zero column in $D(P,T)$.

In all of these cases  $c \in [T] \cup [P]$.
Therefore $(j,c)$ or $(j, N+1-c)$ is a lone star for some $j$.
It follows that $N+1-c$ or $c$ is in $\mathcal{L}_{P,T}$.
\end{proof}

We can now prove that $Z_{P,T}$ is a complete intersection
in types $B$, $C$, and $D$.

\begin{prop}\label{P:complete_intersection_BCD}
Given Schubert symbols $T \preceq P$ for a Grassmannian $X$ of 
Lie type $B$, $C$, or $D$,
the variety $Z_{P,T}$ is a complete intersection in $\P^{N-1}$
cut out by the following polynomials:
\begin{enumerate}[label=\emph{\alph*})]
 \item $f_d - f_c = x_{c+1}x_{N-c} + \ldots + x_dx_{N+1-d}$ 
 if $c$ and $d$ are consecutive elements of $\mathcal{Q}_{P,T}$ such that $d-c \geq 2$, and
 \item $x_c$ if $c \in \mathcal{L}_{P,T}$.
\end{enumerate}
\end{prop}

\begin{proof}

Let $I_{P,T} \subset \C[x_1, \ldots, x_N]$ 
be the ideal generated by the polynomials of types (1) and (2) mentioned
in the statment of this proposition.
\comment{
following polynomials:
\begin{enumerate}
 \item $f_d - f_c = x_{c+1}x_{N-c} + \ldots + x_dx_{N+1-d}$ 
 if $c$ and $d$ are consecutive elements of $\mathcal{Q}_{P,T}$ such that $d-c \geq 2$, and
 \item $x_c$ if $c \in \mathcal{L}_{P,T}$.
\end{enumerate}
By construction,
each of the generators of $I_{P,T}$ is
an irreducible polynomial.}
Note that each of these polynomials is irreducible, and that
by Corollary \ref{C:lin_cuts} and
Proposition \ref{P:lin_cuts_all_types},
no variable $x_i$ appears in multiple generators.

It follows that
$\C[x_1,\ldots,x_N]/I_{P,T}$ is a tensor product
over $\C$
of finitely many integral domains.
Since $\C$ is algebraically closed,
$\C[x_1,\ldots,x_N]/I_{P,T}$ must itself
be an integral domain, by \cite[Lemma 1.5.2]{springer:lin_alg_gps}.
Hence $I_{P,T}$ is a prime ideal.

Let $I'_{P,T}$ be the ideal generated by the polynomials
used to define $Z_{P,T}$: namely,
$\{f_c \mid c \in \mathcal{Q}_{P,T}\} \cup \{x_c \mid c \in \mathcal{L}_{P,T}\}$.
Note that each of the generators of $I_{P,T}$ is
a linear combination of these defining polynomials, and is therefore
contained in $I'_{P,T}$.

On the other hand, 
note that whenever $d-1$ and $d$ are elements of $\mathcal{Q}_{P,T}$,
the polynomial $f_d - f_{d-1} = x_dx_{N+1-d}$ is contained in $I_{P,T}$,
by Lemma \ref{L:consec_cuts}.
Thus if $c<d$ are \emph{any} consecutive elements of $\mathcal{Q}_{P,T}$,
then $f_d-f_c \in I_{P,T}$.
Now, supposing $f_c$ is one of 
the quadratic polynomials defining $Z_{P,T}$, let
$\{0 = c_0 <  c_1 < \ldots < c_s = c\}$ be the complete
list of cuts between $0$ and $c$.
It follows that 
$f_c = f_c - f_0 =
(f_{c_s} - f_{c_{s -1}}) +
(f_{c_{s-1}} - f_{c_{s-2}}) + \ldots +
(f_{c_1}-f_{c_0}) \in I_{P,T}$,
and therefore that
$I'_{P,T} \subset I_{P,T}$.

We have shown that
$I'_{P,T} = I_{P,T}$, and hence that
$Z_{P,T}$ is the zero set
of a prime ideal.
It follows that the polynomials
used to define $I_{P,T}$ also
cut out $Z_{P,T}$ as a complete intersection in $\P^{N-1}$.
\comment{
$I'_{P,T}$ is also a prime ideal.
Thus $I'_{P,T} = \sqrt{I'_{P,T}} = I(Z_{P,T})$.
It follows that
$Z_{P,T} = Z(I_{P,T})$, and
that the set $Z_{P,T}$ is irreducible.}
\end{proof}

\section{Result 2: $\psi(\pi^{-1}(Y_{P,T})) \subset Z_{P,T}$}\label{S:type_D_into}
Let $X := OG(m,2n+2)$ and let $N := 2n+2$.
We would like to show that any vector lying in any subspace $\Sigma \in Y_{P,T}$
satisfies the equations defining $Z_{P,T}$. \comment{use I_{P,T} here?}
The equations involving exceptional cuts are the most difficult to verify,
so we'll address them first.

Let $Y^{\circ}_{P,T}
=X^{\circ}_P(E_{\bull}) \cap X^{\circ}_{T^{\vee}}(E^{\op}_{\bull})$.
It is a dense open subset of $Y_{P,T}$ (see \cite{richardson:intersections}),
so we can restrict our attention to 
$\psi(\pi^{-1}(Y^{\circ}_{P,T}))$.

\begin{prop}\label{P:exceptional}
Consider Schubert symbols $T \preceq P$ for $OG(m,2n+2)$,
and suppose $c\in[1,n]$ is an exceptional cut in $D(P,T)$.
Then $f_{c}(\+w) = 0$ for all $\+w \in \psi(\pi^{-1}(Y^{\circ}_{P,T}))$.
\end{prop}

\begin{proof}
 
Since $c$ is exceptional,
we know $[c+1, n+1] \subset [T] \cap [P]$, $\#P \cap [1,c]+1 = \#T \cap [1,c]$,
and $\type(P) \neq \type(T)$.

\comment{
Note that
\begin{align*}
 E_c &= \langle \+e_1, \ldots, \+e_c\rangle \\
 E_{c}^{\perp} &= \langle \+e_1, \ldots, \+e_{N-c}\rangle \\
 E^{\op}_c &= \langle \+e_{N+1-c}, \ldots, \+e_N\rangle \\
 (E^{\op}_c)^{\perp} &= \langle \+e_{c+1}, \ldots, \+e_N\rangle
\end{align*}
}

Let $\ell = n+1-c$, and
let $E^{(c)} = E_{N-c} / E_{c}$, which we identify with
the span $\langle \+e_{c+1}, \ldots, \+e_{N-c}\rangle$.
Finally, let $\alpha = \#P \cap [1,c]$.

\comment{
Define the orthogonal projection
$\rho: \C^N \to E^{(c)}$,
and let $\Phi_c:U_{c,\alpha} \to OG(\ell,E^{(c)})$ be the morphism from
Proposition \ref{P:max_og_morphism}, which sends
$\Sigma \mapsto \rho(\Sigma \cap E_{N-c})$.

Let $\beta = \#T \cap [N+1-c, N] = m - (\alpha + \ell + 1)$.
Define $\tilde{U}_{c,\beta} = \{ \Sigma \in OG' 
\mid \dim(\Sigma \cap (E^{\op}_{c})^{\perp})=\beta+\ell
\text{ and } \dim(\Sigma \cap E^{\op}_{c}) = \beta\}$.
The map $\Phi_2:\tilde{U}_{c,\beta} \to OG(\ell,2\ell)$ defined by
$\Sigma \mapsto \rho_V(\Sigma \cap (E^{\op}_{c})^{\perp})$
is a morphism, also by Proposition \ref{P:max_og_morphism}.
Both these morphisms restrict to $Y^{\circ}_{P,T} \subset U_{c,\alpha} \cap \tilde{U}_{c,\beta}$.

\vspace{8mm}

\input{max_map.pdf_tex}

\vspace{8mm}
}

Suppose $\Sigma$ is an element of $Y^{\circ}_{P,T}$. 
Since $\Sigma \in X^{\circ}_P$,
we have $\dim(\Sigma \cap E_c) = \alpha$.
Similarly, since $\Sigma \in X^{\circ}_{T^{\vee}}(E^{\op}_{\bull})$,
we have $\dim(\Sigma \cap E^{\op}_c) = m - (\alpha +\ell + 1)$.
Furthermore,  $\dim(\Sigma \cap E_{N-c}) = \alpha + \ell$
and  $\dim(\Sigma \cap E^{\op}_{N-c}) = m - (\alpha +1)$.
Finally,
we know that
\begin{align*}
\dim(\Sigma \cap E^{(c)}) & = \dim(\Sigma  \cap E_{N-c} \cap E^{\op}_{N-c}) \\
& \geq \dim(\Sigma \cap E_{N-c}) + \dim(\Sigma \cap E^{\op}_{N-c}) - m \\
& = (\alpha + \ell) + (m-(\alpha+1)) - m \\
& = \ell -1.
\end{align*}

Therefore we can choose vectors $\+u_1$ through $\+u_m$ spanning
$\Sigma$ such that
\begin{align*}
\+u_i & \in E_c = \langle \+e_1, \ldots, \+e_{c}\rangle& 
\text{for }&  1 \leq i \leq \alpha, \\
\+u_i& \in E^{(c)} = \langle \+e_{c+1}, \ldots, \+e_{N-c}\rangle&
\text{for }& \alpha+2 \leq i \leq \alpha+\ell,\\
\+u_i& \in E^{\op}_c = \langle \+e_{N-c+1}, \ldots, \+e_N\rangle&
\text{for }& \alpha+\ell+2 \leq i \leq m.
\end{align*}

In other words, $\Sigma$ can be represented as the rowspace
of a matrix with the following shape (in the sense that all
entries outside the horizontal arrows are zero):

\[
\left[
  \begin{array}{ccccccccc}
    \multicolumn{3}{l}{\xleftarrow{\hspace{.6cm}}  \+u_1  \xrightarrow{\hspace{.6cm}}} & & & & & & \\
    \multicolumn{3}{c}{\vdots} & & & & & & \\
    \multicolumn{3}{l}{\xleftarrow{\hspace{.6cm}}  \+u_\alpha  \xrightarrow{\hspace{.6cm}}} & & & & & &\\
    \multicolumn{9}{l}{\xleftarrow{\hspace{3.1 cm}}  \+u_{\alpha+1} \xrightarrow{\hspace{3.6 cm}}} \\ 
    & & & \multicolumn{3}{c}{\xleftarrow{\hspace{.5cm}}  \+u_{\alpha+2} \xrightarrow{\hspace{.5cm}}} & & & \\
    & & & \multicolumn{3}{c}{\vdots}  & & & \\
    & & & \multicolumn{3}{c}{\xleftarrow{\hspace{.5cm}}  \+u_{\alpha+\ell}  \xrightarrow{\hspace{.535cm}}} & & & \\
    \multicolumn{9}{l}{\xleftarrow{\hspace{3.15 cm}}  \+u_{\alpha+\ell+1} \xrightarrow{\hspace{3.25 cm}}} \\ 
    & & & & & & \multicolumn{3}{r}{\xleftarrow{\hspace{.38cm}}  \+u_{\alpha+\ell +2}  \xrightarrow{\hspace{.38cm}}} \\
    & & & & & &  \multicolumn{3}{c}{\vdots}  \\
    \coolunder{c}{& \hphantom{aaaaaaaa}&} & \coolunder{2\ell}{& \hphantom{aaaaaaaaa} &} & \multicolumn{3}{r}{\coolunder{c}{\xleftarrow{\hspace{.7cm}}  \+u_{m}  \xrightarrow{\hspace{.7cm}}}} \\   
  \end{array}
\right].
\]
\vspace{10mm}

Furthermore, since $\dim(\Sigma \cap E_{N-c}) = \alpha + \ell$, we can assume
without loss of generality that 
$\+u_{\alpha+1} \in E_{N-c}$.  The matrix with rowspace $\Sigma$ then
has the following shape.

\[
\left[
  \begin{array}{ccccccccc}
    \multicolumn{3}{l}{\xleftarrow{\hspace{.6cm}}  \+u_1  \xrightarrow{\hspace{.6cm}}} & & & & & & \\
    \multicolumn{3}{c}{\vdots} & & & & & & \\
    \multicolumn{3}{l}{\xleftarrow{\hspace{.6cm}}  \+u_\alpha  \xrightarrow{\hspace{.6cm}}} & & & & & &\\
    \multicolumn{6}{l}{\xleftarrow{\hspace{1.83cm}}  \+u_{\alpha+1}  \xrightarrow{\hspace{1.83cm}}} & & & \\
    & & & \multicolumn{3}{c}{\xleftarrow{\hspace{.5cm}}  \+u_{\alpha+2} \xrightarrow{\hspace{.5cm}}} & & & \\
    & & & \multicolumn{3}{c}{\vdots}  & & & \\
    & & & \multicolumn{3}{c}{\xleftarrow{\hspace{.5cm}}  \+u_{\alpha+\ell}  \xrightarrow{\hspace{.535cm}}} & & & \\
    \multicolumn{9}{l}{\xleftarrow{\hspace{3.15 cm}}  \+u_{\alpha+\ell+1} \xrightarrow{\hspace{3.25 cm}}} \\ 
    & & & & & & \multicolumn{3}{r}{\xleftarrow{\hspace{.38cm}}  \+u_{\alpha+\ell +2}  \xrightarrow{\hspace{.38cm}}} \\
    & & & & & &  \multicolumn{3}{c}{\vdots}  \\
    \coolunder{c}{& \hphantom{aaaaaaaa}&} & \coolunder{2\ell}{& \hphantom{aaaaaaaaa} &} & \multicolumn{3}{r}{\coolunder{c}{\xleftarrow{\hspace{.7cm}}  \+u_{m}  \xrightarrow{\hspace{.7cm}}}} \\   
  \end{array}
\right].
\]
\vspace{10mm}

We shall now consider two cases, corresponding to
whether or not $\+u_{\alpha+1}$ is contained
in $E^{\op}_{N-c}$.

{\bf Case 1}: $\+u_{\alpha+1} \in E^{\op}_{N-c}$.

The matrix with rowspace $\Sigma$ then
has the following shape.

\[
\left[
  \begin{array}{ccccccccc}
    \multicolumn{3}{l}{\xleftarrow{\hspace{.6cm}}  \+u_1  \xrightarrow{\hspace{.6cm}}} & & & & & & \\
    \multicolumn{3}{c}{\vdots} & & & & & & \\
    \multicolumn{3}{l}{\xleftarrow{\hspace{.6cm}}  \+u_\alpha  \xrightarrow{\hspace{.6cm}}} & & & & & &\\ 
    & & & \multicolumn{3}{c}{\xleftarrow{\hspace{.5cm}}  \+u_{\alpha+1} \xrightarrow{\hspace{.5cm}}} & & & \\
    & & & \multicolumn{3}{c}{\xleftarrow{\hspace{.5cm}}  \+u_{\alpha+2} \xrightarrow{\hspace{.5cm}}} & & & \\
    & & & \multicolumn{3}{c}{\vdots}  & & & \\
    & & & \multicolumn{3}{c}{\xleftarrow{\hspace{.5cm}}  \+u_{\alpha+\ell}  \xrightarrow{\hspace{.535cm}}} & & & \\
    \multicolumn{9}{l}{\xleftarrow{\hspace{3.15 cm}}  \+u_{\alpha+\ell+1} \xrightarrow{\hspace{3.25 cm}}} \\ 
    & & & & & & \multicolumn{3}{r}{\xleftarrow{\hspace{.38cm}}  \+u_{\alpha+\ell +2}  \xrightarrow{\hspace{.38cm}}} \\
    & & & & & &  \multicolumn{3}{c}{\vdots}  \\
    \coolunder{c}{& \hphantom{aaaaaaaa}&} & \coolunder{2\ell}{& \hphantom{aaaaaaaaa} &} & \multicolumn{3}{r}{\coolunder{c}{\xleftarrow{\hspace{.7cm}}  \+u_{m}  \xrightarrow{\hspace{.7cm}}}} \\   
  \end{array}
\right].
\]
\vspace{10mm}

Note that for $1  \leq \beta \leq \alpha+\ell$,
we have
$u_{\beta,j} = 0$ 
for any $N+1-c \leq j \leq N$.
Thus for $1 \leq \beta \leq \alpha+\ell$ we have
$u_{\alpha+\ell+1,1} \cdot u_{\beta,N} + \ldots + u_{\alpha+\ell+1,c} \cdot u_{\beta,N+1-c} = 0$,
where $u_{i,j}$ is the $j$th coordinate of $\+u_i$.

Because $\Sigma$ is isotropic,
we have $\omega(\+u_{\alpha+\ell+1},\+u_{\beta}) = 0$
for all $1 \leq \beta \leq m$.  In particular,
for $\alpha + \ell + 2 \leq \beta \leq m$, we then have
$u_{\alpha+\ell+1,1} \cdot u_{\beta,N} + \ldots + u_{\alpha+\ell+1,c} \cdot u_{\beta,N+1-c} = 0$.

Finally, let $\+v$ be the orthogonal projection of
$\+u_{\alpha+\ell+1}$ onto $E^{(c)}$.
The span of $\+u_{\alpha+1},\ldots,\+u_{\alpha+\ell+1}$ is a \emph{maximal}
isotropic subspace of $E^{(c)}$, so $\+v$ must be contained in that span.
In particular, $\+v$ is itself an isotropic vector.
Thus 
$u_{\alpha+\ell+1,1} \cdot u_{\alpha+\ell+1,N} + \ldots + u_{\alpha+\ell+1,c} \cdot u_{\alpha+\ell+1,N+1-c} = 0$.
It follows that
$f_c(\+w) = 0$ for any vector $\+w$ in $\Sigma$.

{\bf Case 2}: $\+u_{\alpha+1} \not\in E^{\op}_{N-c}$.

Since
$\dim(\Sigma \cap E^{\op}_{N-c}) = m- (\alpha+1)$,
we may assume $\+u_{\alpha+\ell+1} \in E^{\op}_{N-c}$,
after possibly adding a linear combination of 
$\+u_1, \ldots, \+u_{\alpha+1}$.
Hence there exists a matrix with rowspace $\Sigma$ of
the following shape.

\comment{
By Appendix \ref{A:annoying} (or by Conjecture \ref{C:nonempty})
we know that $\dim(\Sigma \cap 
R_{\alpha+1}
\cap R_{\alpha+\ell+1} )=0$ 
for general $\Sigma$ in $Y^{\circ}_{P,T}$.

Without loss of generality,
we can therefore choose independent vectors
$\+u_{\alpha+1} \in R_{\alpha+1}$ and
$\+u_{\alpha+\ell+1} \in R_{\alpha+\ell+1}$ 
to produce a matrix with the following shape:
} 
\[
\left[
  \begin{array}{ccccccccc}
    \multicolumn{3}{l}{\xleftarrow{\hspace{.6cm}}  \+u_1  \xrightarrow{\hspace{.6cm}}} & & & & & & \\
    \multicolumn{3}{c}{\vdots} & & & & & & \\
    \multicolumn{3}{l}{\xleftarrow{\hspace{.6cm}}  \+u_\alpha  \xrightarrow{\hspace{.6cm}}} & & & & & &\\
    \multicolumn{6}{l}{\xleftarrow{\hspace{1.83cm}}  \+u_{\alpha+1}  \xrightarrow{\hspace{1.83cm}}} & & & \\
    & & & \multicolumn{3}{c}{\xleftarrow{\hspace{.5cm}}  \+u_{\alpha+2} \xrightarrow{\hspace{.5cm}}} & & & \\
    & & & \multicolumn{3}{c}{\vdots}  & & & \\
    & & & \multicolumn{3}{c}{\xleftarrow{\hspace{.5cm}}  \+u_{\alpha+\ell}  \xrightarrow{\hspace{.5cm}}} & & & \\
    & & & \multicolumn{6}{r}{\xleftarrow{\hspace{1.86cm}}  \+u_{\alpha+\ell+1} \xrightarrow{\hspace{1.85cm}}} \\ 
    & & & & & & \multicolumn{3}{r}{\xleftarrow{\hspace{.38cm}}  \+u_{\alpha+\ell +2}  \xrightarrow{\hspace{.38cm}}} \\
    & & & & & &  \multicolumn{3}{c}{\vdots}  \\
    \coolunder{c}{& \hphantom{aaaaaaaa}&} & \coolunder{2\ell}{& \hphantom{aaaaaaaaa} &} & \multicolumn{3}{r}{\coolunder{c}{\xleftarrow{\hspace{.7cm}}  \+u_{m}  \xrightarrow{\hspace{.7cm}}}} \\   
  \end{array}
\right].
\]
\vspace{10mm}

Let $\rho:\C^N \to \C^N$ be the orthogonal projection on to
$E^{(c)}$.
Notice that $\rho(\+u_i) = \+0$ for $i \leq \alpha $ and $i \geq \alpha+\ell+2$.
Define $\+v_i = \rho(\+u_{\alpha+i})$ for $i \in [1,\ell+1]$.
Note that $\+v_i = \+u_{\alpha+i}$ for $i \in [2,\ell]$, but that
\begin{align*}
\+v_1 =& (0,\ldots,0,u_{\alpha+1,c+1},\ldots,u_{\alpha+1,N-c},0,\ldots,0) \\
\text{and } \+v_{\ell+1} =& (0,\ldots,0, u_{\alpha+\ell+1,c+1},\ldots,u_{\alpha+\ell+1,N-c} ,0,\ldots,0),
\end{align*}
where $u_{i,j}$ is the $j$th coordinate of $\+u_i$.

\comment{
Recall the functions $\Phi_c$ and $\Phi^{\op}_c$
which send $\Sigma$ into $OG(\ell,E^{(c)})$,
provided $[c+1,n+1] \subset [\sym(\Sigma)]$ (see
Proposition \ref{P:max_og_morphism} and
Section \ref{S:opposite_map}).
}

Let
\begin{align*}
\Sigma' &= \langle \+v_2, \ldots, \+v_{\ell} \rangle, \\
\Sigma_1 &= \langle \+v_1, \ldots, \+v_\ell \rangle, \text{ and} \\
\Sigma_2 &= \langle \+v_2, \ldots, \+v_{\ell+1} \rangle.
\end{align*}
Both $\Sigma_1$ and $\Sigma_2$
are elements of $OG(\ell,E^{(c)})$, and both contain $\Sigma'$. 
It is a well known fact that there are exactly two
isotropic $\ell$-planes containing a given isotropic $(\ell - 1)$-plane \cite[\S 23.3]{fulton:representation_theory}.
In particular there are exactly two elements of $OG(\ell, E^{(c)})$ that contain $\Sigma'$,
so it remains to verify that 
$\Sigma_1$ and $\Sigma_2$ are indeed the same element.
It is here that the types of $P$ and $T$ become relevant.

Namely, let 
\begin{align*}
P^{(c)} &:= \{p -c : p \in P \cap [c+1,N-c]\} \text{ and}\\ 
T^{(c)} &:= \{t -c : t \in T \cap [c+1,N-c]\}
\end{align*}
be Schubert symbols for $OG(\ell, E^{(c)}) \cong OG(\ell,2\ell)$.
Note that whenever a subspace $\Lambda$ is contained
in an intersection of Schubert cells $X^{\circ}_R$
and $X^{\circ}_{S^{\vee}}$ in $OG(\ell, 2\ell)$,
it must be the case that $S \preceq R$ 
and hence that $\type(\Lambda) = \type(R) = \type(S)$.
In particular, $\Sigma_1 \in X^{\circ}_{P^{(c)}} \subset OG(\ell,2\ell)$
by the definition of $X^{\circ}_{P^{(c)}}$,
so it follows that $\type(\Sigma_1) = \type(P^{(c)})$.
Similarly,  $\Sigma_2 \in X^{\circ}_{T^{{(c)},\vee}} \subset OG(\ell,2\ell)$
by Observation \ref{O:two_wrongs},
and so $\type(\Sigma_2) = \type(T^{(c)})$.

We claim that
$\type(P) = \type(T)$ if and only if
$\type(P^{(c)}) \neq \type(T^{(c)})$.
To see why, note that
\begin{align*}
\#([1,n+1] \setminus P)  &= \#([1,\ell] \setminus P^{(c)}) + (c-\alpha) , \text{ and} \\
 \#([1,n+1] \setminus T) &= \#([1,\ell] \setminus T^{(c)}) + (c-(\alpha+1)).
\end{align*}
Thus $\type(P) + \type(T) \equiv \type(P^{(c)}) + \type(T^{(c)}) + 1 \pmod{2}$.

\comment{
Namely, by Lemma \ref{L:type_of_image}
we have $\type(\Phi_c(\Sigma)) \equiv \type(P) + c + \alpha$.
On the other hand, by Lemma \ref{L:type_of_op_image},
we have $\type(\Phi^{\op}_c(\Sigma)) \equiv \type(T) + c + \alpha + 1$.
It follows that
$\type(P) \neq \type(T)$ if and only if $\type(\Phi_1(\Sigma)) = \type(\Phi_2(\Sigma))$.
}

Since we know that $\type(P) \neq \type(T)$, we can  conclude that
 $\Sigma_1 = \Sigma_2$.
It follows that $\+v_{\ell+1} \in \Sigma_1$.
But $\Sigma_1$ is isotropic, and therefore
$\omega(\+v_1,\+v_{\ell+1})=0$.
Thus, for all $\+w \in \Sigma$, the polynomial
$x_{c+1}x_{N-c} + \ldots + x_{n+1}x_{n+2}$ vanishes, and 
hence $f_{c}(\+w)=0$.
\end{proof}

Having addressed exceptional cuts, we can now
prove a generalization of Lemma \ref{L:type_B_C_into}
for a Grassmannian $X$ of type $B$, $C$, or $D$.

\begin{prop}\label{P:into}
Given Schubert symbols $T \preceq P$ in $\Omega(X)$,
where $X$ is a Grassmannian of type $B$, $C$, or $D$,
the projected Richardson variety $\psi(\pi^{-1}(Y_{P,T}))$
is contained in $Z_{P,T}$.
\end{prop}

\begin{proof}
Fix $\Sigma \in Y_{P,T}$ and $\+w=(w_1,\ldots,w_{N}) \in \Sigma$.
Suppose $c \in \mathcal{Q}_{P,T}$.
We will first show that
$f_{c}(\+w)=0$.

If $c$ (or $N-c$) is a visible cut,
then there exists $j \in [1,m]$ such that
$p_j \leq c < t_{j+1}$ (or $p_j \leq N-c < t_{j+1}$).

Let $W_1 = \langle \+e_1, \ldots, \+e_{p_j} \rangle $ 
and $W_2 =  \langle \+e_{p_j+1}, \ldots, \+e_{N} \rangle $.
Since $\C^N = W_1 \oplus W_2$, there exists
a unique decomposition $\+w = \+w_1 + \+w_2$ with
$\+w_i \in W_i$. Namely, we have $\+w_1 = (w_1, \ldots, w_{p_j}, 0, \ldots, 0)$
and  $\+w_2 = (0, \ldots, 0, w_{p_{j}+1}, \ldots, w_{N})$.
Note that since $w_i = 0$ for $p_j < i < t_{j+1}$,
we have $\+w_2 = (0, \ldots, 0, w_{t_{j+1}}, \ldots, w_{N})$,
and hence $f_{c}(\+w) = (\+w_1,\+w_2)$.

Since $\Sigma \in Y_{P,T}$,
$\dim(\Sigma \cap W_1) \geq j$ and
$\dim(\Sigma \cap W_2) \geq m-j$.
We can therefore write $\Sigma
= \Sigma \cap W_1 \oplus \Sigma \cap W_2$, and
decompose $\+w$ as the sum of vectors in these
subspaces. Since there is only one such decomposition, we must have
$\+w_1 \in \Sigma \cap W_1 $ and $\+w_2 \in \Sigma \cap W_2$.
Since $\Sigma$ isotropic, $f_{c}(\+w) = (\+w_1,\+w_2) = 0$.

If $c=n+1$ (implying we are working in type $B$ or $D$), then
$f_{c} =f_{n+1}$, the inherent quadratic equation,
which vanishes on all isotropic vectors.

Lastly, if $c$ is exceptional, then 
$f_{c}(\+w)=0$ by Proposition \ref{P:exceptional}.

We have shown that the projected
Richardson variety $\psi(\pi^{-1}(Y_{P,T}))$ lies in the zero
set of the polynomial
$f_{c}$, for any element $c \in \mathcal{Q}_{P,T}$.
In particular, it satisfies all the quadratic equations defining $Z_{P,T}$.

We must now show that $\psi(\pi^{-1}(Y_{P,T}))$ satisfies
the linear equation $x_c=0$ for any $c \in \mathcal{L}_{P,T}$.
Suppose $c \in [1,N]$ is a zero column of $D(P,T)$,
and let $\Sigma$ continue to denote an arbitrary
element of $Y_{P,T}$.
Let $j=\max\{i \in [1,m]: p_i < c\}$, and note that
$t_{j+1} > c$.
This time let $W_1 = \langle \+e_1, \ldots, \+e_{c-1} \rangle $ 
and $W_2 =  \langle \+e_{c+1}, \ldots, \+e_{N} \rangle $.
Thus,
$\dim(\Sigma \cap W_1) \geq j $
and $\dim(\Sigma \cap W_2) \geq m-j $,
so $\Sigma \subset W_1 \oplus W_2$,
which is 
the hyperplane defined by $x_c = 0$.

Finally, suppose 
column $d \in [1,N]$ contains a lone star.
In other words $d = p_j$ or $d=t_j$ for some $j$,
and $d-1$ and $d$ are both cuts in $D(P,T)$.
Let $c=\min(d,N+1-d)$, which is in  $[1,\ceil{N/2}]$.
Both $c-1$ and $c$ are also cuts in $D(P,T)$,
so the set
$\psi(\pi^{-1}(Y_{P,T}))$ must
lie in the zero set of the polynomial 
$f_{c} - f_{c-1} = x_cx_{N+1-c} = x_dx_{N+1-d}$.
However, since $\psi(\pi^{-1}(Y_{P,T}))$ is irreducible
(by \cite{richardson:intersections}),
it must lie in $x_d =0$ or $x_{N+1-d}=0$ (or both).

Consider any subspace $\Lambda$ in
the dense subset $X^{\circ}_P(E_{\bull})
\cap X^{\circ}_{T^{\vee}}(E^{\op}_{\bull}) \subset Y_{P,T}$.
\comment{worth giving reference for density? Richardson's paper for example?}
Since $d \in P \cup T$,
it is impossible that $x_d =0$ on all vectors in $\Lambda$.
Therefore, the equation $x_{N+1-d}=0$ must
be satisfied by
$\psi(\pi^{-1}(Y_{P,T}))$.
\end{proof}

\section{Result 3: $Z_{P,T} \subset \psi(\pi^{-1}(Y_{P,T}))$}\label{S:type_BC_onto}

Let $X := IG_\omega(m,\C^N)$ be a Grassmannian of type $B$, $C$, or $D$.
Given Schubert symbols $P$ and $T$ in $\Omega(X)$, we write $P \to T$
whenever
\begin{enumerate}[label=\emph{\roman*})]
 \item $T \preceq P$;
 \item $p_i \leq t_{i+1}$ for all $i$, unless $p_i=n+2$ and $t_{i+1} = n+1$ in type $D$; and
 \item if $p_i = t_{i+1}$, then $p_i$ is \emph{not} a cut in $D(P,T)$.
\end{enumerate}

When $p_i > t_{i+1}$, we say $D(P,T)$ has a \emph{$2 \times 2$ square}.
Thus, the second condition says that $D(P,T)$ has no $2 \times 2$ squares, except that
a single $2 \times 2$ square in the central columns is permitted in type $D$.
When $p_i=t_{i+1}$ for some $i$ and $p_i \in \mathcal{C}_{P,T}$,
then $(i+1,t_{i+1})$ is  a lone star in $D(P,T)$ (in fact, so is $(i,p_i)$),
and hence $N+1-p_i \in \mathcal{L}_{P,T}$.
If we are working in type $B$ or $C$, then $N+1-p_i$
must be a zero column, which yields
the equivalent definition of $P \to T$ given
in \cite{buch:quantum_pieri}.

If we are working in type $D$, then there are other possibilities
involving exceptional cuts.
For example, if $p_i = t_{i+1} \in \{n+1,n+2\}$,
then $n+1$ will be an exceptional cut, causing
both $(i,p_i)$ and $(i+1,t_{i+1})$ to be lone stars.
It follows that when $P \to T$, the diagram $D(P,T)$
cannot have exactly $3$ stars in the central columns $n+1$ and $n+2$,
which yields
the alternative type $D$ definition of $P \to T$ given
in \cite[\S 5.2]{buch:quantum_pieri}.

The relation $P \to T$ is important because
it characterizes precisely when the map $\psi: \pi^{-1}(Y_{P,T}) \to Z_{P,T}$
is a birational isomorphism.  In particular 
\cite[Proposition 5.1]{buch:quantum_pieri} says the following:
\begin{prop}\label{P:birational_crit}
Given a Grassmannian $X$ of Lie type $B$, $C$, or $D$,
the map $\psi: \pi^{-1}(Y_{P,T}) \to Z_{P,T}$ is a birational isomorphism
if and only if the Schubert symbols $P$ and $T$ satisfy the relation $P \to T$.
\end{prop}

A more detailed proof of Proposition \ref{P:birational_crit} can be found in \cite[\S 8]{ravikumar:thesis}.
In this section we prove the following proposition.
\begin{prop}\label{P:shrink_existence}
 Given Schubert symbols $T \preceq P$ for a Grassmannian $X$ of type $B$, $C$, or $D$, 
 there exists a Schubert symbol $\tilde{P}$
 such that:
 \begin{enumerate}
  \item $T \preceq \tilde{P} \preceq P$,
  \item $Z_{\tilde{P},T} = Z_{P,T}$, and
  \item $\tilde{P} \to T$.
 \end{enumerate}
\end{prop}

We prove Proposition \ref{P:shrink_existence} by explicitly
constructing the Schubert symbol $\tilde{P}$.
An immediate consequence is that
for any $T \preceq P$ in $\Omega(X)$,
the Richardson variety $Y_{P,T}$ must contain
a smaller Richardson variety $Y_{\tilde{P},T}$ such that
$\psi:\pi^{-1}(Y_{\tilde{P},T}) \to Z_{\tilde{P},T} = Z_{P,T}$ 
is a birational isomorphism (by Proposition \ref{P:birational_crit}).
It follows that $Z_{P,T} \subset \psi(\pi^{-1}(Y_{P,T}))$.
By combining this observation with Proposition \ref{P:into}, which states $Z_{P,T} \supset \psi(\pi^{-1}(Y_{P,T}))$, 
we have a proof of Theorem \ref{T:intro_main_theorem}.
\comment{We shall prove Proposition \ref{P:shrink_existence} by explicitly
constructing the Schubert symbol $\tilde{P}$.}  
We mention that Proposition \ref{P:shrink_existence} is not the only way to prove $Z_{P,T} \subset \psi(\pi^{-1}(Y_{P,T}))$.
For example, instead of ``lowering'' $P$ we could ``raise'' $T$.  In fact, the construction
we give for $\tilde{P}$ does precisely that when carried out on the rotated diagram $D(\bar{T},\bar{P})$.

\subsection{Constructing a ``Smaller'' Schubert Symbol P'}\label{S:instant_shrink_b}
\comment{
We prove Proposition \ref{P:shrink_existence} by explicitly constructing
the Schubert symbol $P'$.  We carry out this construction
for a Grassmannian $X$ of Lie type $B$, $C$, or $D$ (we will use this construction
in Section \ref{S:type_D_onto} to prove the type $D$
version of Proposition \ref{P:shrink_existence}).
}
\comment{
Recall the following subsets of $[1,N]$:
\begin{enumerate}
\item $\mathcal{C} := \mathcal{C}_{P,T} = \{c \in [1,N] \mid c \text{ is a cut in } D(P,T)\}$.
\item $\mathcal{L} := \mathcal{L}_{P,T} = \{c \in [1,N] \mid x_c=0 \text{ on } Z_{P,T}\}$.
\end{enumerate}  
}
Given Schubert symbols $T \preceq P$, 
define the set ${P}' = \{{p}'_1 , \ldots , {p}'_m\}$ 
as follows:

If $p_i < t_{i+1}$, then
\begin{equation}\label{bshrink1}\tag{\text{$\diamondsuit$}}
{p}'_i = p_i.
\end{equation}  

\hphantom{}

On the other hand, 
if $p_i \geq t_{i+1}$ and $t_{i+1}-1 \not\in \mathcal{C}_{P,T}$, then
\begin{equation}\label{bshrink2}\tag{\text{$\clubsuit$}}
{p}'_i = t_{i+1}.
\end{equation}

Finally,
if $p_i \geq t_{i+1}$ \emph{and} $t_{i+1}-1 \in \mathcal{C}_{P,T}$, then
\begin{equation}\label{bshrink3}\tag{\text{$\heartsuit$}}
{p}'_i = \max\{c \in [t_i, t_{i+1}-1] \mid c \not\in \mathcal{L}_{P,T}\}. 
\end{equation} 

Note that since
$t_i$ cannot be in $\mathcal{L}_{P,T}$,
the set $\{c \in [t_i, t_{i+1}-1] \mid c \not\in \mathcal{L}_{P,T}\}$ 
is nonempty. Thus $p'_i$ is well-defined.

The following property of $P'$ follows from its construction and 
Corollary \ref{C:no_lin_in_PT}.

\begin{obs}\label{O:lin}
For any $1 \leq i \leq m$, $p'_i \not\in \mathcal{L}_{P,T}$.
\end{obs}

\begin{proof}
 If $p'_i$ is defined by Case \eqref{bshrink1}
 or Case \eqref{bshrink2}, then $p'_i \in P \cup T$,
 so by Corollary \ref{C:no_lin_in_PT} it is not in $\mathcal{L}_{P,T}$.
 If it is defined by Case \eqref{bshrink3}, then by its construction
 it cannot be in $\mathcal{L}_{P,T}$.
\end{proof}

We also have the following observation, which follows directly from the construction
of $P'$.

\begin{obs}\label{O:cut}
 If $p'_i$ is produced by Case \eqref{bshrink1}
 or Case \eqref{bshrink3} of the previous construction,
then $[p'_i, t_{i+1}-1] \subset \mathcal{C}_{P,T}$.
 Furthermore, if $p'_i < t_{i+1} -1$, then 
 $[p'_i+1, t_{i+1}-1] \subset \mathcal{L}_{P,T}$.
\end{obs}

We claim that $P'$ is a Schubert symbol.

\begin{lemma}\label{L:bshrink_is_schubert_symbol}
 $P'$ is a Schubert symbol.
\end{lemma}

\begin{proof}

We will first show that $p'_i < p'_{i+1}$ for
 $1 \leq i \leq m-1$.  
 By our construction of $P'$, we have
 $p'_i \leq t_{i+1} \leq p'_{i+1}$
 for each $i$.  Thus we only need to consider the
 case that $p'_{i+1} = t_{i+1}$.
 If  $p'_{i+1} = t_{i+1}$ then
 $p'_{i+1}$ follows Case \eqref{bshrink1} or Case \eqref{bshrink3},
 so by Observation \ref{O:cut},
 $t_{i+1} \in \mathcal{C}_{P,T}$.
 But then $(i+1,t_{i+1})$ is a lone star in $D(P,T)$,
 so $t_{i+1} - 1 \in \mathcal{C}_{P,T}$ as well (by Corollary \ref{C:lin_cuts}). 
Thus $p'_i$ follows Case \eqref{bshrink1} or Case \eqref{bshrink3},
and we have $p'_i < t_{i+1}$.
It follows that
$p'_i < p'_{i+1}$ for all $1 \leq i \leq m-1$,
and in particular that
$P'$ consists of $m$ distinct integers.

We still have to check that $P'$
satisfies the isotropic condition.
Suppose on the contrary that 
$p'_i + p'_j = N+1$ 
for some $i$ and $j$ in $[1,m]$.
By Observation \ref{O:cut}, if
$p'_i$ is not a cut in $D(P,T)$,
then $p'_i = t_{i+1}$.
Similarly, if
$p'_{j}$
is not a cut in $D(P,T)$,
then $p'_j = t_{j+1}$.

We therefore have three possible cases to
consider, each of which results in a contradiction.

{\bf Case 1}: Both $p'_i$ and $p'_j$
are cuts in $D(P,T)$.

In this case
$N-p'_j = p'_i-1$ is also a cut,
and by Lemma \ref{L:consec_cuts} either $p'_i$
or $p'_j$ is in $\mathcal{L}_{P,T}$.
But neither $p'_i$
nor $p'_j$ can be in $\mathcal{L}_{P,T}$, by
Observation \ref{O:lin}.

{\bf Case 2}: Exactly one of $\{p'_i, p'_j\}$ is a cut in $D(P,T)$.

We will assume without loss of generality 
that $p'_j$ is a cut in $D(P,T)$.
We then have $p'_i-1 \in \mathcal{C}_{P,T}$, as in the previous case,
implying that $p'_i$
cannot follow Case \eqref{bshrink2}.
However, since $p'_i \not\in \mathcal{C}_{P,T}$,
Observation \ref{O:cut} implies that $p'_i$
\emph{does} follow Case \eqref{bshrink2}, a contradiction.

{\bf Case 3}: Neither $p'_i$ nor $p'_j$ is a cut in $D(P,T)$.

In this case, $p'_i = t_{i+1}$
and $p'_j = t_{j+1}$.  But
$t_{i+1} + t_{j+1} \neq N+1$, since
$T$ is a Schubert symbol, so once
again we arrive at a contradiction.
\end{proof}

\subsection{Types B and C}
If $X$ is a Grassmannian of type $B$ or $C$, we set $\tilde{P} := P'$
in order to prove Proposition \ref{P:shrink_existence}.
The type $D$ case of Proposition \ref{P:shrink_existence}
will be addressed in Section \ref{S:type_D_onto}.

Since $T \preceq P$
if and only if $T \leq P$,
it is clear by our construction
that $t_i \leq \tilde{p}_i \leq p_i$ for
$1 \leq i \leq m$,
and hence
$T \preceq \tilde{P} \preceq P$.  Thus $\tilde{P}$
satisfies condition (1) of
Proposition \ref{P:shrink_existence}.
We will need the following lemma to 
verify the remaining conditions.

\begin{lemma}\label{L:same_cuts_type_BC}
  The diagrams $D(P,T)$ and $D(\tilde{P},T)$
 have the same cuts.
\end{lemma}

\begin{proof}
Given an integer $c \in [1,N]$,
if $p_i \leq c < t_{i+1}$, then $\tilde{p}_i = p_i \leq c  < t_{i+1}$.
Therefore  $\mathcal{C}_{P,T} \subset \mathcal{C}_{\tilde{P},T}$.
 On the other hand, if $\tilde{p}_i < t_{i+1}$, then
 $\tilde{p}_i$ is defined by \eqref{bshrink1} or \eqref{bshrink3}.
 We then have
 $[\tilde{p}_i, t_{i+1} -1] \subset \mathcal{C}_{P,T}$,
 by Observation \ref{O:cut}.  Therefore
 $\mathcal{C}_{P,T} \supset \mathcal{C}_{\tilde{P},T}$.
 \end{proof}

\begin{proof}[Proof of Proposition \ref{P:shrink_existence}]
We show that $\tilde{P}$ satisfies conditions (2) and (3). 

{\em $\tilde{P}$ satisfies (2): $Z_{P,T} = Z_{\tilde{P},T}$.}

  By Lemma \ref{L:same_cuts_type_BC},
 the diagrams $D(P,T)$ and $D(\tilde{P},T)$
 have the same cuts.
 We still have to prove that
  $\mathcal{L}_{P,T} = \mathcal{L}_{\tilde{P},T}$.
 Suppose $c \in \mathcal{L}_{P,T}$.
 Then by Proposition \ref{P:simple_lone_star}, either
 $p_i = t_i = N+1-c$ for some $i$,
 or  $c$ is a zero column of $D(P,T)$.
  If $p_i = t_i = N+1-c$ for some $i$,
  then $p_i < t_{i+1}$, so $\tilde{p}_i$
  must follow Case \eqref{bshrink1} of
  the construction.
  Thus $\tilde{p}_i = t_i = N+1-c$,
 so $c \in \mathcal{L}_{\tilde{P},T}$.
 On the other hand if
 $c$ is a zero column in $D(P,T)$,
 then $p_i < c < t_{i+1}$ for some $i$.
Once again we have $\tilde{p}_i = p_i$,
 so $c$ is a zero column in $D(\tilde{P},T)$,
 and therefore an element of $\mathcal{L}_{\tilde{P},T}$.

Now suppose $c \in \mathcal{L}_{\tilde{P},T}$.
If $\tilde{p}_i < c < t_{i+1}$ for some $i$, then
$\tilde{p}_i$ satisfies \eqref{bshrink1}
or \eqref{bshrink3}.  Either way,
we must have $c \in \mathcal{L}_{P,T}$.
On the other hand if 
$(i,N+1-c)$ is a lone star in $D(\tilde{P},T)$
for some $i$, then
$\tilde{p}_i = t_i = N+1-c$.
Thus, $\tilde{p}_i \in \mathcal{C}_{\tilde{P},T} = \mathcal{C}_{P,T}$.
In this case $(i,t_i)$ is a lone star in $D(P,T)$,
and $c \in \mathcal{L}_{P,T}$.

{\em $\tilde{P}$ satisfies (3): $\tilde{P} \to T$.}
   
Since $\tilde{p}_i \leq t_{i+1}$ for $1 \leq i \leq m-1$,
the diagram $D(\tilde{P},T)$ has no $2 \times 2$ squares.
If $\tilde{p}_i = t_{i+1}$ for some $i$, then $t_{i+1} - 1$
 is not a cut in $D(P,T)$. By Lemma \ref{L:same_cuts_type_BC} it is also not a 
 cut in $D(\tilde{P},T)$,
 so column $N+1-\tilde{p}_i$
 cannot be a zero column in $D(\tilde{P},T)$.
 \end{proof}

\subsection{Type D}\label{S:type_D_onto}

\comment{
Let $X := OG(m,2n+2)$ be a type $D$ Grassmannian,
and let $T \preceq P$ be Schubert symbols in $\Omega(X)$.
In this section we construct a Schubert symbol $\tilde{P}$
satisfying 
$T \preceq \tilde{P} \preceq P$,
$Z_{P,T} = Z_{\tilde{P},T}$,
and  $\tilde{P} \to T$.
}
\comment{
Recall the following subsets of $[1,N]$:
\begin{enumerate}
\item $\mathcal{C} := \mathcal{C}_{P,T} = \{c \in [1,N] \mid c \text{ is a cut in } D(P,T)\}$.
\item $\mathcal{L} := \mathcal{L}_{P,T} = \{c \in [1,N] \mid x_c=0 \text{ on } Z_{P,T}\}$.
\end{enumerate}
}

Let $X := OG(m,2n+2)$ be a type $D$ Grassmannian, let $N := 2n+2$,
and let $T \preceq P$ be Schubert symbols in $\Omega(X)$.
Unfortunately the Schubert symbol $P'$ constructed in Section \ref{S:instant_shrink_b} fails to satisfy the
conditions of Proposition \ref{P:shrink_existence}, as the following example illustrates.

\begin{example}
Consider
$OG(2,6)$, and let $T = \{1,3\}$
and $P = \{5,6\}$.  The projected Richardson variety
$Z_{P,T}$ is the quadric hypersurface
of $\P^{5}$ consisting of all isotropic lines in $\C^6$.
\comment{
\[
\left(
  \begin{array}{cccccc}
    * & * & * & * & * & 0  \\
    0 & 0 & * & * & * & *  \\
  \end{array}
\right).
\]
}
Since $P' = \{3,6\}$,
and the variety $Z_{P',T}$ satisfies the additional linear equation $x_4 = 0$,
so $Z_{P,T} \neq Z_{P',T}$.
Furthermore, $P' \not\to T$, since $p_1 = t_2 = 3$ is a cut in $D(P',T)$--the exceptional center cut:
\[
\left(
  \begin{array}{cccccc}
    * & * & * & 0 & 0 & 0  \\
    0 & 0 & * & * & * & *  \\
  \end{array}
\right).
\]
\end{example}

\vskip 5mm

We therefore define the set
$\tilde{P} = \{\tilde{p}_1 , \ldots , \tilde{p}_m\}$ 
as follows:

If $p_i < t_{i+1}$, then
\begin{equation}\label{shrink1}\tag{\text{$\diamondsuit$}}
\tilde{p}_i = p_i.
\end{equation}

On the other hand, 
if $p_i \geq t_{i+1}$ and $t_{i+1}-1 \not\in \mathcal{C}$, then
\begin{equation}\label{shrink2}\tag{\text{$\clubsuit$}}
\tilde{p}_i = \left\{
\begin{array}{lr}
t_{i+1} & \text{ if } t_{i+1} \not\in \{n+1,n+2\},\\
N+1-t_{i+1} & \text{ if } t_{i+1} \in \{n+1,n+2\}.\\
\end{array}
\right. 
\end{equation}

Finally,
if $p_i \geq t_{i+1}$ \emph{and} $t_{i+1}-1 \in \mathcal{C}$, then
\begin{equation}\label{shrink3}\tag{\text{$\heartsuit$}}
\tilde{p}_i = \max\{c \in [t_i, t_{i+1}-1] \mid c \not\in \mathcal{L}\}. 
\end{equation} 

Note that since
$t_i$ cannot be in $\mathcal{L}$,
the set $\{c \in [t_i, t_{i+1}-1] \mid c \not\in \mathcal{L}\}$ 
is nonempty. Thus $\tilde{p}_i$ is well-defined.

Recall that $\iota$ is the permutation of 
$\{1, \ldots, 2n+2\}$ that interchanges $n+1$ and $n+2$
and leaves all other numbers fixed.
We make the following observation:

\begin{obs}
 $\tilde{P}$ is equal to $\iota P'$ or $P'$.  Moreover,
 $\tilde{P} = \iota P'$ if and only if
 there exists an element $\tilde{p}_i$ defined by Case \eqref{shrink2}
 and  $t_{i+1} \in \{n+1,n+2\}$.
\end{obs}

Given a type $D$ Schubert symbol $R$,
the set $\iota R$ is also a Schubert symbol,
since the isotropic condition is preserved.
We therefore have the following corollary.

\begin{cor}
 $\tilde{P}$ is a Schubert symbol.
\end{cor}

The following observations are 
exact restatements
of Observations \ref{O:lin}
and \ref{O:cut} for type $D$.
We give a brief proof of the first,
whereas the second follows directly
from the construction of $\tilde{P}$.

\begin{obs}\label{O:lin_D}
For any $1 \leq i \leq m$, $\tilde{p}_i \not\in \mathcal{L}_{P,T}$.
\end{obs}

\begin{proof}
 If $\tilde{p}_i$ is defined by Case \eqref{shrink1},
 or if it is defined by Case \eqref{shrink2} and $t_{i+1} \not\in \{n+1,n+2\}$,
 then $\tilde{p}_i \in P \cup T$.
 By Corollary \ref{C:no_lin_in_PT} it is not in $\mathcal{L}_{P,T}$.
 If it is defined by Case \eqref{shrink3}, then by its construction
 it cannot be in $\mathcal{L}_{P,T}$.
 
Finally, suppose $\tilde{p}_i$ is defined by Case \eqref{shrink2}
 and $t_{i+1} \in \{n+1,n+2\}$. If $\tilde{p}_i = N+1-t_{i+1} \in \mathcal{L}_{P,T}$,
then $t_{i+1}-1 \in \mathcal{C}_{P,T}$,
contradicting the assumption that $\tilde{p}_i$ is
defined by Case \eqref{shrink2}.
\end{proof}

\begin{obs}\label{O:cut_D}
 If $\tilde{p}_i$ is produced by Case \eqref{shrink1}
 or Case \eqref{shrink3} of the previous construction,
then $[\tilde{p}_i, t_{i+1}-1] \subset \mathcal{C}_{P,T}$.
 Furthermore, if $\tilde{p}_i < t_{i+1} -1$, then 
 $[\tilde{p}_i+1, t_{i+1}-1] \subset \mathcal{L}_{P,T}$.
\end{obs}

We will now proceed to prove that
$T \preceq P$,
$\tilde{P} \to T$,
$Z_{P,T} = Z_{\tilde{P},T}$, and
$\tilde{P} \preceq P$, in that order.
We begin with an important lemma.

\comment{

\begin{obs}\label{O:center_zero}
If $\tilde{p}_i = t_{i+1}$ for some $1 \leq i \leq m-1$,
then $\tilde{p}_i \not\in \{n+1,n+2\}$.
\end{obs}

\begin{proof}
If $\tilde{p}_i = t_{i+1}$ then $\tilde{p}_i$
is defined by \eqref{shrink2} and $t_{i+1} \not\in \{n+1,n+2\}$.
Therefore $\tilde{p}_i \not\in \{n+1,n+2\}$.
\end{proof}
}

\begin{lemma}\label{L:same_type}
If $n+1 \in [T] \cap [\tilde{P}]$ then $\type(T) = \type(\tilde{P})$.
\comment{
 If $\type(T) \in \{0,1\}$, then
 $\type(\tilde{P}) \in \{\type(T),2\}$.}
\end{lemma}

\begin{proof}
 Suppose $\tilde{p}_i \in \{n+1,n+2\}$ for some $i$
 and $t_j \in \{n+1,n+2\}$ for some $j$.
 
 Since $\tilde{p}_k \geq t_k > n+2$ for all $k >j$,
 we have $i \leq j$.
 Furthermore, $\tilde{p}_{k-1} \leq t_{k}  < n+1$ for all $k < j$,
 so $i \geq j-1$.
 Thus we either have $i = j$ or $i+1 = j$.
 
 {\bf Case 1:} $i = j$. 
 
 We claim that $t_i = \tilde{p}_i \in \{n+1,n+2\}$.
 Suppose on the contrary that $t_i = n+1$ and $\tilde{p}_i = n+2$.
 Since $t_{i+1} \neq n+1$, $\tilde{p}_i$ is \emph{not} defined
 by \eqref{shrink2}.  Therefore by Observation \ref{O:cut_D},
 $\tilde{p}_i \in \mathcal{C}_{P,T}$.  It follows that
 both $n$ and $n+2$ are in $\mathcal{Q}_{P,T}$, so
 by Lemma \ref{L:consec_cuts}, either $n+1$ or $n+2$
 is in $\mathcal{L}_{P,T}$.  But neither $t_i$ nor
 $\tilde{p}_i$ can be in $\mathcal{L}_{P,T}$,
 by Corollary \ref{C:no_lin_in_PT} and Observation \ref{O:lin_D}.
 To avoid this contradiction, we must have $t_i = \tilde{p}_i \in \{n+1,n+2\}$.
Thus, $n+1$ is a visible cut in $D(\tilde{P},T)$,
and $\#([1,n+1] \cap T) = \#([1,n+1] \cap \tilde{P})$.
In other words, $\type(T) = \type(\tilde{P})$.
 
 \comment{
 If $t_i = n+1$, then
 suppose for the sake of contradiction that $p_{i-1} \geq n+2$.
 Note that $n = t_i-1$ is not a cut in $D(P,T)$,
 and therefore $\tilde{p}_{i-1}$ is defined by \eqref{shrink2}.
 It follows that $\tilde{p}_{i-1} = n+2$, and hence
 that $\tilde{p}_i > n+2$, contradicting the assumption
 that $\tilde{p}_i \in \{n+1,n+2\}$.
 Therefore, we must have $p_{i-1} \leq t_{i} = n+1$,
 and hence $D(P,T)$ must have an exceptional center cut.
 It follows that $(i,t_i)$ is a lone star in $D(P,T)$,
 and therefore that $n+2 \in \mathcal{L}_{P,T}$.
 By Observation \ref{O:lin_D}, $\tilde{p}_i = n+1$.
 If on the other hand $t_i = n+2$, then clearly $\tilde{p}_i = n+2$ as well,
 since $t_i \leq \tilde{p}_i$.
It follows that $t_i = \tilde{p}_i \in \{n+1,n+2\}$.
Thus, $n+1$ is a visible cut in $D(\tilde{P},T)$,
and $\#([1,n+1] \cap T) = \#([1,n+1] \cap \tilde{P})$.
In other words, $\type(T) = \type(\tilde{P})$.
}

{\bf Case 2:} $i+1=j$.

If $\tilde{p}_i = t_{i+1}$ then $\tilde{p}_i$
is defined by \eqref{shrink2} and $t_{i+1} \not\in \{n+1,n+2\}$.
Hence $\tilde{p}_i \not\in \{n+1,n+2\}$, a contradiction.
It follows that
$\tilde{p}_i \neq t_{i+1}$.
If $\tilde{p}_i = n+1$ and $t_{i+1} = n+2$,
then $n+1$ is a visible cut in $D(\tilde{P},T)$,
as seen on the left side of Figure \ref{F:center_crossings}.
On the other hand, if 
$\tilde{p}_i = n+2$ and $t_{i+1} = n+1$ as seen on the right side,
then $\#([1,n+1] \cap T) = \#([1,n+1] \cap \tilde{P})+2$,
so once again $\type(T) = \type(\tilde{P})$.
\begin{figure}
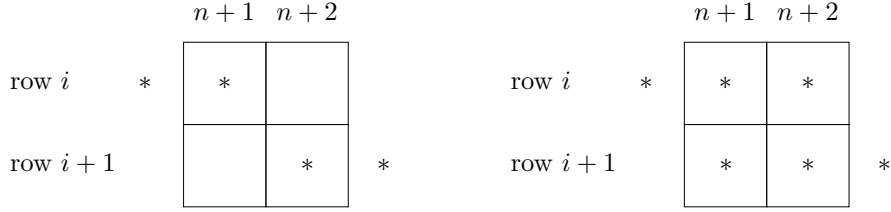
\caption{Center columns of $D(\tilde{P},T)$ in Lemma \ref{L:same_type}.}
\label{F:center_crossings}
\end{figure}
\end{proof}

 Lemma \ref{L:same_type} says that  $\{\type(\tilde{P}), \type(T)\} \neq \{0,1\}$.  Since
 we also know that $T \leq \tilde{P}$, by construction,
 we have the following immediate
 corollaries. 
 
 \begin{cor}\label{C:T_less_than}
 $T \preceq \tilde{P}$.
\end{cor}

\begin{proof}
 The condition that $\{ \type(\tilde{P}),\type(T) \} = \{0,1\}$
 is necessary for $T \not\prec \tilde{P}$.
\end{proof}

Corollary \ref{C:T_less_than} says that
$\tilde{P}$ satisfies the first half of
condition (1) of Proposition \ref{P:shrink_existence}.
We prove the second half in Proposition \ref{P:type_D_cond_3}.

\begin{cor}\label{C:no_excep}
The diagram $D(\tilde{P},T)$ has no exceptional cuts,
except possibly the center cut $n+1$.
\end{cor}

\begin{proof}
 The condition that $\{ \type(\tilde{P}),\type(T) \} = \{0,1\}$
 is necessary for the existence of an exceptional cut other than the center cut.
\end{proof}

In fact, any non-central exceptional cut in $D(P,T)$ becomes a visible cut in $D(\tilde{P},T)$.
However, we will only need this fact in the special case that the
exceptional cut is less than $n+1$:

\begin{lemma}\label{L:what_happens_to_exceptional_cuts}
 If $1 \leq c \leq n$ is an exceptional cut in $D(P,T)$,
 then $c$ is a visible cut in $D(\tilde{P}, T)$.
\end{lemma}

\begin{proof}
 Since $c$ is an exceptional cut in $D(P,T)$,
 we have that $\#([1,c] \cap T) = \#([1,c] \cap P) +1$.
 In other words there exists
 exactly one integer $i \in [1,m]$ such that
 $t_i \leq c < p_i$.

 Since row $i$ is the only row crossing
 from column $c$ to column $c+1$, we must have $c+1 \leq t_{i+1}$.
 Furthermore, since $n+1 \in [T]$, we must have $t_{i+1} \leq n+2$.
 We will show that $\tilde{p}_i \leq c$, and hence that $c$ is a
 visible cut in $D(\tilde{P},T)$.
 We divide the rest of our argument into two cases.
 
 {\bf Case 1:} $c+1 = t_{i+1}$.
 
 Since $c \in \mathcal{C}_{P,T}$, $\tilde{p}_i$ is defined
 by \eqref{shrink3}.  It follows that $\tilde{p}_i \leq c$.

 {\bf Case 2:} $c+1 < t_{i+1} \leq n+2$.
 
 We claim that  $p_i \geq t_{i+1}$.
 To see why, suppose for the sake of contradiction
 that  $p_i < t_{i+1}$. 
 Then $p_i$ is a visible cut in $D(P,T)$, and
 $\#([1,p_i] \cap T) = \#([1,p_i] \cap P)$.
 If $p_i < n+1$, then since $[p_i+1, n+1] \subset [P] \cap [T]$
 and $\type(T) \neq \type(P)$ (which follow from the
 fact that $c$ is an exceptional cut), we have $T \not\prec P$,
 a contradiction.
  If $p_i = n+1$, then $\#([1,n+1] \cap T) = \#([1,n+1] \cap P)$,
 so $\type(P) = \type(T)$, contradicting the assumption
 that $c$ is an exceptional cut.
 Finally, if $p_i > n+1$, then
 $t_{i+1} > n+2$, contradicting the
 assumption that $n+1 \in [T]$.
 It follows that $p_i \geq t_{i+1}$, as claimed.
  
 Now note that any $d \in [c+1, t_{i+1}-1]$ must
 be an exceptional cut in $D(P,T)$,
 since $\#([1,d] \cap T) = \#([1,d] \cap P) + 1$
 (in particular  row $i$ is the only row
 crossing from column $d$ to column $d+1$).
 Furthermore, for each $d \in [c+1, t_{i+1}-1]$,
 we have
 $N+1-d \in T$, since $d \not\in T$ and $[c+1,n+1] \in [T]$.
 It follows that column $N+1-d$ contains a lone star in $D(P,T)$,
 and hence that $d \in \mathcal{L}_{P,T}$.
 Since $[c+1, t_{i+1}-1] \subset \mathcal{L}_{P,T}$,
 it follows that
 $\tilde{p}_i$ is defined by \eqref{shrink3} and
 that $\tilde{p}_i \leq c$.
\end{proof}

The following proposition will also be needed.

\begin{prop}\label{P:cuts_preserved}
  $\mathcal{C}_{P,T} \cup \{n+1\} = \mathcal{C}_{\tilde{P},T} \cup \{n+1\}$.
\end{prop}

\begin{proof}
By Corollary \ref{C:no_excep}, there are no exceptional cuts
in $D(\tilde{P},T)$.  If $c$ is a visible cut in $D(\tilde{P},T)$,
then $\tilde{p}_i \leq c < t_{i+1}$ for some $1 \leq i \leq m-1$.
If $\tilde{p}_i$ is defined by \eqref{shrink1} or \eqref{shrink3}
then by Observation \ref{O:cut_D}, $c \in \mathcal{C}_{P,T}$.
Otherwise $\tilde{p}_i$ is defined by \eqref{shrink2} and
$\tilde{p}_i = c = n+1$.
Since all visible cuts in $\mathcal{C}_{\tilde{P},T}$
are contained in  $\mathcal{C}_{P,T} \cup \{n+1\}$,
the same is true for apparent cuts.
It follows that
$\mathcal{C}_{\tilde{P},T} \subset \mathcal{C}_{P,T} \cup \{n+1\}$.

On the other hand, if $p_i < t_{i+1}$ then
$\tilde{p}_i = p_i < t_{i+1}$.
Thus any visible cut $c \in \mathcal{C}_{P,T}$ is also
contained in $\mathcal{C}_{\tilde{P},T}$.
If $c < n+1$ is an exceptional cut in $D(P,T)$, 
then by Lemma \ref{L:what_happens_to_exceptional_cuts},
$c \in \mathcal{C}_{\tilde{P},T}$.
It follows that
$\mathcal{C}_{P,T} \subset \mathcal{C}_{\tilde{P},T} \cup \{n+1\}$.

\comment{
then
let $\alpha = \#([1,c] \cap P)$.  Recall
that $t_{\alpha+1} \leq c$ and $p_{\alpha+1} > c$.
Since $c+1 \in [T]$, $t_{\alpha+2} = c+1$
or $t_{\alpha+2} = N-c$.  

If $t_{\alpha+2} = c+1$, then $\tilde{p}_{\alpha+1} < c+1$
since $c$ is a cut in $D(P,T)$.
On the other hand, if $t_{\alpha+2} = N-c$,
then we must have $c = n$ and 
$t_{\alpha+2} = n+2$, since $[c+1,n+1] \subset [T]$.
But then $p_{\alpha+1} = n+2$, since $\type(P) \neq \type(T)$ and
$p_{\alpha+1} \in \{n+1,n+2\}$.  Thus $n+1 \in \mathcal{L}_{P,T}$,
so $\tilde{p}_{\alpha+1} < n+1 = c+1$ once again.
 Therefore in either case
$c$ is a visible cut in $D(\tilde{P},T)$.}
\end{proof}

We can now prove that $\tilde{P}$ satisfies
condition (3) of Proposition \ref{P:shrink_existence}.

\begin{prop}\label{P:type_D_cond_1}
 $\tilde{P} \to T$.
\end{prop}

\begin{proof}
By Corollary \ref{C:T_less_than},
$T \preceq \tilde{P}$.
The construction of $\tilde{P}$ ensures that
 $\tilde{p}_i$ is never greater than $t_{i+1}$,
 except possibly when $t_{i+1} = n+1$ and $\tilde{p}_i = n+2$.
 Furthermore, if $\tilde{p}_i = t_{i+1}$, then $t_{i+1} -1$
 is not a cut in $D(P,T)$.  Therefore,
 $t_{i+1}$ cannot be a cut in $D(P,T)$ either,
 since that would make $(i+1,t_{i+1})$ a lone star in $D(P,T)$
 and hence make $t_{i+1} -1$ a cut in $D(P,T)$, by
 Proposition \ref{P:lin_cuts_all_types}.
It follows that $t_{i+1}$ is not a cut
in $D(\tilde{P},T)$, by
Proposition \ref{P:cuts_preserved}.
\end{proof}

We will now show that
 $Z_{P,T} = Z_{\tilde{P},T}$ by
 examining the diagrams $D(P,T)$
 and $D(\tilde{P},T)$.
 Recall that the quadratic equations defining $Z_{P,T}$
 are entirely determined by $\mathcal{C}_{P,T}$,
 the set of cuts in the $D(P,T)$.
 
Notice that
$\mathcal{C}_{P,T}$ is not equal to $\mathcal{C}_{\tilde{P},T}$
 in general, because our construction of $\tilde{P}$ adds the center cut
 $n+1$ whenever $p_i > t_{i+1} = n+2$.  By Proposition \ref{P:cuts_preserved},
 that is the only change in the cut set, and the addition of
 the center cut does not alter the equations in these cases.

 In fact, by Proposition \ref{P:cuts_preserved}, $Z_{P,T}$
and $Z_{\tilde{P},T}$ satisfy the same quadratic
equations: namely $\{f_{c} \mid c \in (\mathcal{C}_{P,T} \cap [1,n]) \cup \{n+1\} \}$.
The following proposition shows that they 
satisfy the same linear equations as well.

\begin{prop}
  $\mathcal{L}_{P,T} = \mathcal{L}_{\tilde{P},T}$.
\end{prop}

\begin{proof} 
 Suppose $c \in \mathcal{L}_{\tilde{P},T}$.
 If $c$ is a zero column in $D(\tilde{P},T)$,
 then $\tilde{p}_i < c < t_{i+1}$ for some $i$.
Whether $\tilde{p}_i$ satisfies \eqref{shrink1},
 or \eqref{shrink3}, we then have
 $c \in \mathcal{L}_{P,T}$.
 
 If $c$ is not a zero column in the diagram,
 then column $N+1-c$ contains a lone star in $D(\tilde{P},T)$.
 If neither $N+1-c$ nor $N-c$ are exceptional cuts in $D(\tilde{P},T)$
then we must have $t_i = \tilde{p}_i = N+1-c$ for some $i$,
by Proposition \ref{P:simple_lone_star}.
Thus $t_i \in \mathcal{C}_{\tilde{P},T}$,
and hence by Proposition \ref{P:cuts_preserved}, $t_i \in \mathcal{C}_{P,T}$.
It follows that  $(i,t_i)$ is a lone star in $D(P,T)$
and $c \in \mathcal{L}_{P,T}$.
On the other hand if either
$N+1-c$ or $N-c$ 
is an exceptional
cut in $D(P,T)$ then by Corollary \ref{C:no_excep},
$n+1$ must be that exceptional cut.
It follows that
$c \in \{n+1,n+2\}$. Since $c \in \mathcal{L}_{\tilde{P},T}$,
we have $n$ and $n+2$ in $\mathcal{C}_{\tilde{P},T}$,
and hence in $\mathcal{C}_{P,T}$ by Proposition \ref{P:cuts_preserved}.
By Lemma \ref{L:consec_cuts},
either $n+1$ or $n+2$ (that is,
either $c$ or $N+1-c$) must be in $\mathcal{L}_{P,T}$.
But column $N+1-c$ contains a lone star 
in $D(\tilde{P},T)$, and hence
$N+1-c \in T \cup \tilde{P}$.
Thus by Corollary \ref{C:no_lin_in_PT} and 
Observation \ref{O:lin_D}, it is impossible
for $N+1-c$ to be in $\mathcal{L}_{P,T}$.
Therefore $c \in \mathcal{L}_{P,T}$.
It follows that
 $\mathcal{L}_{\tilde{P},T} \subset \mathcal{L}_{P,T}$.

 On the other hand, if $c \in \mathcal{L}_{P,T}$, then $c-1$ and $c$ are both
 cuts in $D(P,T)$.  By Proposition \ref{P:cuts_preserved}
 they are both cuts in $D(\tilde{P},T)$ as well,
 and therefore either $c$ or $N+1-c$ is
 in $\mathcal{L}_{\tilde{P},T}$ by Lemma \ref{L:consec_cuts}.
 If $N+1-c \in \mathcal{L}_{\tilde{P},T}$,
 then we have shown that $N+1-c \in \mathcal{L}_{P,T}$ as well.
 In that case by Corollary \ref{C:no_lin_in_PT}, $D(P,T)$ does not
 have a lone star in column $N+1-c$, so $c$ must be a zero
 column in $D(P,T)$. In other words
 $p_i < c < t_{i+1}$ for some $i$.
 But then $\tilde{p}_i = p_i$, so
 $c$ is a zero column in $D(\tilde{P},T)$ as well.
\end{proof}

We have condition (2) Proposition \ref{P:shrink_existence} as an immediate corollary:
\begin{cor}\label{C:type_D_cond_2}
 $Z_{P,T} = Z_{\tilde{P},T}$.
\end{cor}

\comment{
\subsection{Containment of the Type D Shrunken Richardson Variety}
}

Finally, we prove that $\tilde{P} \preceq P$,
and hence that the Richardson variety $Y_{\tilde{P},T}$
is indeed contained in $Y_{P,T}$.
Our proof will require the following somewhat technical lemma.

\begin{lemma}\label{L:technical}
Suppose $\type(T) \neq \type(P)$, and that
there exists an integer $1 \leq c \leq n$ such that
$c \not\in [T]$ and
$[c+1,n+1] \subset [P] \cap [T] \cap [\tilde{P}]$.
Then $c \not\in [P] \cap [\tilde{P}]$
\end{lemma}

\begin{proof}
 Suppose for the sake of contradiction that
 $c \in [P] \cap [\tilde{P}]$.  We divide our
 argument into four cases.
 
 {\bf Case 1:} $c \in \tilde{P}$ and $N+1-c \in P$.
 
 In this case $\tilde{p}_i = c$ for some $i \in [1,m]$,
 and $p_{j} = N+1-c$  for some $j \in [1,m]$.

 Since $c \not\in T$, $\tilde{p}_i$ is
 defined by \eqref{shrink1} or \eqref{shrink3}. Hence
$\tilde{p}_i \in \mathcal{C}_{P,T}$ by Observation \ref{O:cut_D}.
 Since  $p_j - 1 = N-\tilde{p}_i$ is also a cut in $D(P,T)$,
 it follows that  $(j,p_j)$ is a lone star in $D(P,T)$,
 and therefore that $\tilde{p}_i \in \mathcal{L}_{P,T}$,
 contradicting Observation \ref{O:lin_D}.

 {\bf Case 2:} $c \in P$ and $N+1-c \in \tilde{P}$.

  In this case $p_i = c$ for some $i \in [1,m]$,
 and $\tilde{p}_{j} = N+1-c$  for some $j \in [1,m]$.

 Since $N+1-c \not\in T$, $\tilde{p}_j$ is
 defined by \eqref{shrink1} or \eqref{shrink3}. Hence
$\tilde{p}_j \in \mathcal{C}_{P,T}$ by Observation \ref{O:cut_D}.
 Since  $p_i -1= N-\tilde{p}_j$ is also a cut in $D(P,T)$,
 it follows that  $(i,p_i)$ is a lone star in $D(P,T)$,
 and therefore that $\tilde{p}_j \in \mathcal{L}_{P,T}$,
 contradicting Observation \ref{O:lin_D}.

 {\bf Case 3:} $c \in \tilde{P} \cap P$.
 
 In this case $p_i = c$ for some $i$.
 
 Let $\ell := n+1-c$.
 Since $[c+1,n+1] \subset [P]$, Lemma \ref{L:another_condition_for_maximal_projection}
 tells us that $[p_{i+1}, p_{i+\ell}] \subset [c+1,N-c]$.
 We will show that $\tilde{p}_{i+1} < c+1$ and that
 $\tilde{p}_{i+\ell+1} > N-c$,
 contradicting the assumption that
 $[c+1,n+1] \subset \tilde{P}$ and hence
 the assumption that
 $\#([c+1,N-c] \cap \tilde{P}) = \ell$
 (by Lemma \ref{L:another_condition_for_maximal_projection}),
 since in this case there can be at most
 $\ell -1$ elements of $\tilde{P}$ contained
 in the interval $[c+1,N-c]$.
 
We will first show that $\tilde{p}_{i+1} < c+1$.
Since $c \in \tilde{P}$, 
it must be the case that $\tilde{p}_j = c$ for some $j \geq i$.
We will show that $j = i+1$.

Note that if $p_i = \tilde{p}_i < t_{i+1}$ then
$\#([1,c] \cap T) = \#([1,c] \cap P)$, and hence
$T \not\prec P$ (since $\type(T) \neq \type(P)$), 
contradicting the assumption that $T \prec P$.

Also note that if $p_i = \tilde{p}_i = t_{i+1}$,
then $t_{i+1} = c$, contradicting the assumption
the $c \not\in [T]$.

Therefore,  $\tilde{p}_j = c$ for some $j > i$.
Furthermore,
$\tilde{p}_j$ is defined by \eqref{shrink3} since 
$c < t_{j+1}$,
and hence $\tilde{p}_j \in \mathcal{C}_{P,T}$.

However, neither $c$ nor $N-c$ is a visible cut in $D(P,T)$,
since that would imply $T \not\prec P$ by Corollary \ref{C:ways_to_think_about_critical_windows}.
It follows that $c$ is an exceptional cut in $D(P,T)$.

Since $c$ is an exceptional cut,
and since $p_i = c$,
row $i+1$ must be the only row crossing from column $c$ to column $c+1$,
by Corollary \ref{C:ways_to_think_about_critical_windows}.
In other words
$t_{i+1} < c < p_{i+1}$,
and $t_{i+2} > c$.
Thus $j = i+1$, and $\tilde{p}_{i+1} = \tilde{p}_j = c < c+1$.

It remains to show that
 $\tilde{p}_{i+\ell+1} > N-c$.

 Since  $[c+1,n+1] \subset [T]$, and since
 $t_{i+1} < c < t_{i+2}$,
 Lemma \ref{L:another_condition_for_maximal_projection}
 tells us that $[t_{i+2}, t_{i+\ell+1}] \subset [c+1,N-c]$.
Furthermore, 
row $i+\ell+1$ must cross from column $N-c$ to column $N-c+1$,
since $N-c+1 \leq p_{i+\ell+1}$.

In fact, $t_{i+\ell+1} < N-c+1 < p_{i+\ell+1}$,
since $p_{i+\ell+1} \neq N-c+1$ (due to the fact that $c \in P$).
Therefore $N-c+1$ is not a visible cut in $D(P,T)$.
It is not even an apparent cut, since
$t_i < c = p_i$.
Finally, $N+1-c$ is not an exceptional cut, since $c \not\in [T]$.
Thus $N+1-c \not\in \mathcal{C}_{P,T}$.

Note that $t_{i+\ell+2} > N-c+1$.
Thus if $\tilde{p}_{i+\ell+1}$ is defined by \eqref{shrink2},
then $\tilde{p}_{i+\ell+1} > N-c+1$.
Furthermore, if $\tilde{p}_{i+\ell+1}$ is defined by \eqref{shrink1}
or \eqref{shrink3}, then since  $N+1-c \not\in \mathcal{C}_{P,T}$
and $t_{i+\ell+2} > N-c+1$, we must again have
$\tilde{p}_{i+\ell+1} > N-c+1$.

 {\bf Case 4:} $N+1-c \in \tilde{P} \cap P$.
 
 As in the previous case, let $\ell := n+1-c$.
 We have $p_j = N+1-c$ for some $j$.
 Since $[c+1,n+1] \subset [P]$, we have 
 $[p_{j-\ell}, p_{j-1}] \subset [c+1,N-c]$
 by Corollary \ref{C:ways_to_think_about_critical_windows}.
 We will show that $\tilde{p}_j > N-c$
 and $\tilde{p}_{j-\ell} < c+1$,
 contradicting the assumption that
 $[c+1,n+1] \subset [\tilde{P}]$.

 We will first show that $\tilde{p}_j > N-c$.
 Suppose for the sake of contradiction that
 $\tilde{p}_k = N+1-c$ for some $k > j$.
 Then since $N+1-c \not\in T$, 
 $\tilde{p}_k$ is defined by \eqref{shrink1}
 or \eqref{shrink3}, and is therefore
 a cut in $D(P,T)$ by Observation \ref{O:cut_D}.
 Since $c \not\in [T]$, it follows that $\tilde{p}_k$
 is not an exceptional cut.
 Since $t_k \leq \tilde{p}_k = p_j < p_k$,
 it follows that 
 $\tilde{p}_k$
 is not a visible cut either.
 
But  $\tilde{p}_k$ is a cut, and hence
$c-1$ must be a visible cut in $D(P,T)$.
Since $c \not\in T$, $c$ must be a zero column in $D(P,T)$.
But then $c$ is a visible cut,
so by Corollary \ref{C:ways_to_think_about_critical_windows}
$T \not\prec P$, a contradiction.
It follows that $k=j$, and hence that
$\tilde{p}_j = N+1-c > N-c$.

It remains to show that $\tilde{p}_{j-\ell} < c+1$.
Since $p_j = \tilde{p}_j = N+1-c$, and since $N+1-c \not\in T$,
it must be the case that
$\tilde{p}_j$ is defined by \eqref{shrink1}, and therefore that
$p_j$ is a visible cut in $D(P,T)$.
It follows that row $j$ is the only row crossing from column
$N-c$ to column $N-c+1$, and hence that
$c$ and $N-c$ are exceptional cuts in $D(P,T)$
by Corollary \ref{C:ways_to_think_about_critical_windows}.

By Lemma \ref{L:what_happens_to_exceptional_cuts},
$c$ is a visible cut in $D(\tilde{P},T)$.
It follows that
$\tilde{p}_{j-\ell} \leq c$, since
$t_{j-\ell} \leq c$.
\end{proof}

\comment{

Our proof will require a somewhat technical lemma.
We say the 
$i^{\text{th}}$ row of the diagram $D(P,\tilde{P})$ is a \emph{center crossing}
if $\tilde{p}_i \leq n+1 < p_i$.
Note that this definition does not require
$\tilde{P} \prec P$.

\begin{lemma}\label{L:center_crossings}
Suppose there exists an integer $c \in [1,n]$ such that
$[c+1,n+1] \subset [\tilde{P}] \cap [P]$ and
$\#[1,c] \cap \tilde{P} = \#[1,c] \cap P$, and
that $c$ is the maximal integer in $[1,n]$
satisfying these properties.
If, in addition, there is at least one
center crossing in $D(P,\tilde{P})$, then
we have $[c+1,n+1] \subset [T]$.
\end{lemma}

\begin{proof} 
We will show that each $\tilde{p}_i \in [c+1,N-c]$ follows  Case \eqref{shrink2}
of the construction of $\tilde{P}$.

The conditions satisfied by the integer $c$ guarantee that no
integer $d \in [c+1,N-c]$ is a
cut in $D(P,\tilde{P})$, with the possible exception of the center cut $d=n+1$.
However, the possiblity of a center cut is precluded by the additional
contraint that there is at least one center crossing in $D(P,\tilde{P})$.

It follows that no
integer $d \in [c+1,N-c]$ is a
cut in $D(P,T)$, since any cut in $D(P,T)$
is also a cut in $D(P,\tilde{P})$.
Because $D(P,T)$
contains no cuts in the interval $[c+1,N-c]$,
\comment{
also need $c$ maximal to preclude \eqref{shrink1} elements
}
we have
$\tilde{p_i}$ is defined by Case \eqref{shrink2} 
for each $\tilde{p_i} \in [c+1,N-c]$.
In other words $[c+1,n+1] \subset [T]$.
\end{proof}
}

We can now prove that $\tilde{P}$ satisfies 
the second half of condition (1) of Proposition \ref{P:shrink_existence}.

\begin{prop}\label{P:type_D_cond_3}
 $\tilde{P} \preceq P$.
\end{prop}

\begin{proof}

 Note that $\tilde{P} \leq P$ by construction.  Assuming $\tilde{P} < P$, 
 suppose $\tilde{P} \not\prec P$, for the sake of contradiction.  
 It follows that
 $\type(\tilde{P}) \neq \type(P)$, and that there
 exists an integer $c \in [1,n]$ such that
$[c+1,n+1] \subset [\tilde{P}] \cap [P]$ and
$\#([1,c] \cap \tilde{P}) = \#([1,c] \cap P)$.

{\bf Case 1:} $n+1 \not\in [T]$.

Since $n+1 \in [\tilde{P}]$, we have $\tilde{p}_i \in \{n+1,n+2\}$
for some $i$.
Since $n+1 \not\in [T]$, $\tilde{p}_i$ must be defined by \eqref{shrink1}
or \eqref{shrink3}.
By Observation \ref{O:cut_D}, $\tilde{p}_i \in \mathcal{C}_{P,T}$.
If $\tilde{p}_i = n+1$, this means $n+1 \in \mathcal{C}_{P,T}$.
If $\tilde{p}_i = n+2$, then also we have
 $n+1 \in \mathcal{C}_{P,T}$.
 To see why, note that both $n$ and $n+1$ are in
 $\mathcal{Q}_{P,T}$, so
by Lemma \ref{L:consec_cuts}, either $n+1$
or $n+2$ is in $\mathcal{L}_{P,T}$.

If $n+1$ is an exceptional center cut in $D(P,T)$,
then $p_j = n+2 < t_{j+1}$ for some $j$,
since $n+1 \not\in [T]$.  In this case,
$\tilde{p}_j = p_j = n+2$, and hence
$n+1$ is a visible cut in $D(P,\tilde{P})$.
On the other hand if
 $n+1$ is a visible cut in $D(P,T)$, then
it must be a visible cut in $D(P,\tilde{P})$
as well.
It follows that
$\#([1,n+1] \cap P) = \#([1,n+1] \cap \tilde{P})$.
In other words $\type(P) = \type(\tilde{P})$,
a contradiction.

{\bf Case 2:} $n+1 \in [T]$.

Since $n+1 \in [T] \cap [\tilde{P}]$, Lemma \ref{L:same_type}
implies $\type(T) = \type(\tilde{P})$.  Since we are assuming
$\type(\tilde{P}) \neq \type(P)$, this means
$\type(T) \neq \type(P)$.
We can therefore invoke Lemma \ref{L:technical} $n-c$ times to
deduce that $[c+1,n+1] \subset [T]$.

We divide the remainder of the proof into three subcases,
depending on the number of rows crossing from column $c$ to column $c+1$
of $D(P,T)$.

{\bf Case 2a:}
$\#([1,c] \cap T) = \#([1,c] \cap P)$.

Since $\type(T) \neq \type(P)$ and
$[c+1,n+1] \subset [T] \cap [P]$,
we have $T \not\preceq P$, a contradiction.

 {\bf Case 2b:}
 $\#([1,c] \cap T) = \#([1,c] \cap P) + 1$.

 In other words there is exactly one integer $i$
 such that $t_i \leq c < p_i$.
 Furthermore, $c$ is an exceptional cut in $D(P,T)$,
 since $[c+1,n+1] \subset [T]$ and $\type(T) \neq \type(P)$.
 Therefore by Lemma \ref{L:what_happens_to_exceptional_cuts},
 $c$ is a visible cut in $D(\tilde{P},T)$.
 This means $\tilde{p}_i \leq c < t_{i+1}$.
 But $c < p_i$,
 so we have $\tilde{p}_i \leq c < p_i$,
 contradicting the assumption that
  $\#([1,c] \cap \tilde{P}) = \#([1,c] \cap P)$.

 {\bf Case 2c:}
 $\#([1,c] \cap T) \geq \#([1,c] \cap P) + 2$.

Let $i$ be the smallest integer such that $t_i \leq c  < p_i$.
Note that $t_{i+1} \leq c < p_{i+1}$.
Since $\tilde{p}_i \leq t_{i+1} \leq c$, it follows that
$\#([1,c] \cap \tilde{P}) > \#([1,c] \cap P)$,
again contradicting the assumption that $\#([1,c] \cap \tilde{P}) = \#([1,c] \cap P)$.
 \comment{
 Since $\type(\tilde{P}) \neq \type(P)$, there are an
 odd number of center crossings in $D(P,\tilde{P})$, and
 therefore at least one.
Thus, by Lemma \ref{L:center_crossings}, we have $[c+1,n+1] \subset [T]$.

We shall now consider three cases, each of which leads to
a contradiction.

{\bf Case 1}:
$\#[1,c] \cap T = \#[1,c] \cap P$.

Since $[c+1,n+1] \subset [T]$, we know that  $\type(T) \in \{0,1\}$.
By Lemma \ref{L:same_type} we then have $\type(T) = \type(\tilde{P})$.
Since $\type(\tilde{P}) \neq \type(P)$, it follows
that $T \not\preceq P$, a contradiction.

{\bf Case 2}:
$\#[1,c] \cap T = \#[1,c] \cap P + 1$.

In this case $c$ is an exceptional cut in $D(P,T)$,
since $[c+1,n+1] \subset [T] \cap [P]$ and $\type(T) \neq \type(P)$.
It follows that $c+1 \not\in \tilde{P}$,
contradicting the fact
that $[c+1,N-c]$ is a minimal critical window
in $D(P,\tilde{P})$ (see Proposition \ref{P:alt_bruhat} for the definition
of \emph{critical window}).

{\bf Case 3}:
$\#[1,c] \cap T \geq \#[1,c] \cap P + 2$.

Let $i$ be the smallest integer such that $t_i < c+1  \leq p_i$.
Since $\tilde{p}_i \leq t_{i+1} < c+1$, it follows that
$\# [1,c] \cap \tilde{P} > \# [1,c] \cap P$,
contradicting the assumption that $\# [1,c] \cap \tilde{P} = \# [1,c] \cap P$.
}
\end{proof}

Combining Propositions  \ref{P:type_D_cond_1} and \ref{P:type_D_cond_3}, 
and Corollaries \ref{C:T_less_than} and \ref{C:type_D_cond_2}, we see
that the Schubert symbol $\tilde{P}$ satisfies all three conditions 
of Proposition \ref{P:shrink_existence}. By the discussion at the beginning
of Section \ref{S:type_BC_onto}, this completes the proof of 
Theorem \ref{T:intro_main_theorem}.

\section{The Grothendieck Ring}\label{S:groth}

In this section we summarize some facts about $K$-theory which will be used
in Sections \ref{S:triples} and \ref{S:pieris}.  Further
details can be found in \cite{fulton:intersection} or \cite{brion:lectures}.

Given an algebraic variety $X$, let
$K^0(X)$ denote the Grothendieck ring of algebraic vector bundles on $X$.
Let $K_0(X)$ denote the Grothendieck group
of coherent $\cO_X$-modules, which is a module over $K^0(X)$.
Both the ring structure of $K^0(X)$ and the module structure of $K_0(X)$
are defined by tensor products.
A closed subvariety $Z \subset X$ has a \emph{Grothendieck class} $[\cO_Z] \in K_0(X)$
defined by its structure sheaf.
If $X$ is nonsingular, the map $K^0(X) \to K_0(X)$ sending a vector
bundle to its sheaf of sections is an isomorphism,
and we write $K(X) := K_0(X) \cong K^0(X)$, which we refer to as
the \emph{Grothendieck ring of $X$}.

A morphism of varieties $f: X \to Y$, defines a pullback ring homomorphism
$f^*:K^0(Y) \to K^0(X)$ by pullback of vector bundles.  
If $f$ is proper, then there exists a pushforward group
homomorphism $f_*: K_0(X) \to K_0(Y)$.
Both these maps are functorial with respect
to composition of morphisms.
The \emph{projection formula} says that $f_*$ is a
$K^0$-module homomorphism, in the sense that
\begin{equation}\label{E:proj}
f_*(f^*\mathcal{A}\cdot\mathcal{B}) = \mathcal{A}\cdot f_*\mathcal{B}
\end{equation}
where $\mathcal{A} \in K^0(Y)$ and $\mathcal{B} \in K_0(X)$.
If $X$ is a complete variety, then the sheaf Euler characteristic map
$\chi_X:K(X) \to K(\text{point}) = \Z$  is defined to be the pushforward along
the morphism $X \to \text{point}$.

We need the following well known fact (see e.g. \cite[Lemma 2.2]{buch:cominuscule}).
\begin{lemma}\label{L:fundamental_lemma_of_intersection_theory}
Let $X$ be a non-singular variety and let $Y$ and $Z$ be closed varieties
of $X$ with Cohen-Macaulay singularities. Assume that each component of $Y \cap Z$
has dimension $\dim(Y) + \dim(Z) - \dim(X)$. Then $Y \cap Z$ is Cohen-Macaulay and
$[\cO_Y] \cdot [\cO_Z] = [\cO_{Y \cap Z}]$ in $K(X)$.
\end{lemma}

Finally we recall some facts about the 
$K$-theory of the projective space $\P^{N-1}$.
Let $h \in K(\P^{N-1})$ be  the class of a hyperplane. 
Then $h^j$ is the class of a codimension $j$ linear subvariety,
$2h-h^2$ is the class of a quadric hypersurface,
and $K(\P^{N-1}) = \Z[h]/{(h^N)}$.
The sheaf euler characteristic $\chi_{\P^{N-1}}: K(\P^{N-1}) \to \Z$
is determined by $h^j \mapsto 1$ for $1 \leq j < N$.

\section{Computing Triple Intersection Numbers}\label{S:triples}
Let $X := IG_\omega(m,\C^{N})$ be a Grassmannian of type $B$, $C$, or $D$.
In this section we calculate the $K$-theoretic Pieri-type 
triple intersection numbers 
\begin{equation}\label{E:triple_int_triple_section}
\chi_X([\cO_{X_P}] \cdot [\cO_{X^T}] \cdot [\cO_{X_{(r)}}]),
\end{equation}
where $T \preceq P$ are Schubert symbols in $\Omega(X)$ and $X_{(r)} \subset X$ is a special
Schubert variety, which we will define shortly for each type of Grassmannian.
Our technique relies on the projection formula to move our calculation to the $K$-theory of projective space.
A similar technique was used for the type $A$ Grassmannian in \cite{buch:cominuscule}.

Recall the projections $\pi$ and $\psi$ from $IF := IF_\omega(1,m,\C^N)$
to $X$ and $Z := IG_\omega(1,\C^N)$ respectively.
Since the $K$-theoretic pushforward is functorial with respect to composition,
the following diagram commutes:

\begin{center}
\begin{tikzpicture}[node distance = 2cm]
  \node (F) {$K(IF)$};
  \node (G) [below of=F] {$K(X)$};
  \node (P) [right of=F] {$K(Z)$};
  \node (pt) [right of=G] {$\Z$};
  \draw[->] (F) -- node[left]{$\pi_*$} (G);
  \draw[->] (F) -- node[above]{$\psi_*$}(P);
  \draw[->] (F) -- node[right]{$\chi_{IF}$}(pt);
  \draw[->] (G) -- node[below]{$\chi_X$}(pt);
  \draw[->] (P) -- node[right]{$\chi_Z$}(pt);
  \end{tikzpicture}
\end{center}

\begin{lemma}\label{L:cohomologically_trivial_transport}
Let $X := IG_{\omega}(m,\C^N)$ be a Grassmannian of type $B$, $C$, or $D$.
Suppose there exists a Schubert variety $W \subset Z := IG_\omega(1,\C^N)$ such that
$\pi(\psi^{-1}(W)) = X_{(r)}$.  We then have
\begin{equation}
 \chi_X([\cO_{X_P}] \cdot [\cO_{X^T}] \cdot [\cO_{X_{(r)}}])
 = \chi_{Z}([\cO_{Z_{P,T}}] \cdot [\cO_W]),
\end{equation}
for any Schubert symbols $T \preceq P$.
\end{lemma}

\begin{proof}
 Since $\pi^{-1}(X_P)$ and $\pi^{-1}(X^T)$ are Schubert varieties in $IF$,
$\pi^{-1}(Y_{P,T})$ is a Richardson variety in $IF$ (see e.g. \cite{brion:lectures}).
By \cite[4.5]{knutson:projections} or \cite[3.3]{billey:singularities}, the projection
$\psi: \pi^{-1}(Y_{P,T}) \to Z_{P,T}$ is cohomologically trivial, in the sense that
$\psi_*[\cO_{\pi^{-1}(Y_{P,T})}] = [\cO_{Z_{P,T}}]$.
Since $\pi$ is flat, it follows that \[\psi_*\pi^*[\cO_{Y_{P,T}}] = [\cO_{Z_{P,T}}] \in K(Z).\]
Similarly, $\pi: \psi^{-1}(W) \to X_{(r)}$ is cohomologically trivial and $\pi$ is flat, so we have
\[\pi_*\psi^*[\cO_{W}] = [\cO_{X_{(r)}}] \in K(X).\]
It is known that all Schubert varieties have rational singularities \cite{mehta:normality}.
Therefore by Lemma \ref{L:fundamental_lemma_of_intersection_theory} and two applications of
the projection formula, we have:
\begin{align*}
\chi_{X}([\cO_{X_P}] \cdot [\cO_{X^T}] \cdot [\cO_{X_{(r)}}]) &= \chi_{X}([\cO_{Y_{P,T}}] \cdot \pi_*\psi^*[\cO_{W}]) \\
&= \chi_{IF}(\pi^*[\cO_{Y_{P,T}}] \cdot \psi^*[\cO_{W}]) \\
&= \chi_{Z}(\psi_*\pi^*[\cO_{Y_{P,T}}] \cdot [\cO_{W}]) \\
&= \chi_{Z}([\cO_{Z_{P,T}}] \cdot [\cO_{W}]). \\ 
\end{align*}
\end{proof}

By Lemma \ref{L:cohomologically_trivial_transport}, the projected Richardson variety $Z_{P,T}$
assumes a crucial role in the calculation of \eqref{E:triple_int_triple_section}.
Let $q$ be the number of quadratic equations and let $l$ be the number of linear equations
defining $Z_{P,T}$ as a complete intersection in $\P^{N-1}$ (from Sections \ref{S:prelim} and \ref{S:type_D_equations}
we know that $q = \#\{c \in \mathcal{Q}_{P,T} : c>0 \text{ and } c-1 \not\in \mathcal{Q}_{P,T}\}$ and $l = \#\mathcal{L}_{P,T}$).
We now calculate 
\eqref{E:triple_int_triple_section} when $X$ has Lie type
$C$, $B$, or $D$, in that order, finally presenting
a unified treatment in Corollary \ref{C:trip_all}.

\subsection{Type C}
Let $X = SG(m,2n)$.  Note that $SG(1,2n)$, the image of $\psi$, is equal to $\P^{2n-1}$.
The codimension $r$ special Schubert variety $X_{(r)}$ is defined by
\[
X_{(r)} := \{\Sigma \in X: \dim(\Sigma \cap E_{2n-m-r+1}) \geq 1\},
\]
for $1 \leq r \leq 2n-m$.
In other words $X_{(r)} = \pi(\psi^{-1}(W))$,
where $W := \P(E_{2n-m-r+1})$.  Note that $W$ is a linear subvariety of $\P^{2n-1}$
and therefore a Schubert variety.

\comment{
consider the two-step symplectic flag variety $SF := SF(1,m;2n)=\{(L,\Sigma): L \subset \Sigma \} \subset 
SG(1,2n) \times X$, with projections
$\pi:SF \to X$ and $\psi:SF \to SG(1,2n)=\P^{2n-1}$.
\[
\begin{CD}
SF @>\psi>> \P^{2n-1}\\
@VV{\pi}V\\
X
\end{CD}
\]
}

We have the following formula.
\begin{prop}\label{P:trip_C}
Let $X := SG(m,2n)$.
Given $T \preceq P$ in $\Omega(X)$ and $1 \leq r \leq 2n-m$, 
we have
\begin{equation*}
\chi_{X}([\cO_{X_P}] \cdot [\cO_{X^T}] \cdot [\cO_{X_{(r)}}])  = \sum_{j = 0}^{2n-m-r-l-q}{\binom{q}{j}(-1)^{j}(2)^{q-j}},
\end{equation*}
where we define $\binom{q}{j}$ to be zero for $j > q$.
\end{prop}

\begin{proof}
By Lemma \ref{L:cohomologically_trivial_transport}, the triple intersection number is equal to
$\chi_{\P^{2n-1}}([\cO_{Z_{P,T}}] \cdot [\cO_{W}])$.
Since $Z_{P,T}$ is a complete intersection defined by $l$ linear and $q$ quadratic polynomials,
we have
$[\cO_{Z_{P,T}}] = h^l h^q (2-h)^q$.
Since $[\cO_{W}] = h^{m-1+r}$, it follows that
\begin{align*}\label{E:type_C_K_class}
 [\cO_{Z_{P,T}}] \cdot [\cO_{W}] &=  h^{m+r+l+q-1}(2-h)^{q} \\
  &=  \sum_{j = 0}^{2n-m-r-l-q}h^{m+r+l+q-1}{\binom{q}{j}(-h)^{j}(2)^{q-j}},
\end{align*}
where we define $\binom{q}{j}$ to be zero for $j > q$.
Taking sheaf Euler characteristic yields the desired triple intersection formula.
\end{proof}

\subsection{Type B}
Let $X := OG(m,2n+1)$.
Let $Q := OG(1,2n+1)$ denote the quadric hypersurface of isotropic lines in $\P^{2n}$ with inclusion
$\iota: Q \hookrightarrow \P^{2n}$.

We describe the Schubert varieties relative to $E_{\bull}$ for
the odd dimensional quadric $Q$.
For  $0 \leq j \leq 2n-1$ there is exactly one codimension $j$
Schubert variety $Q_{(j)} \subset Q$, defined by
 \begin{equation*}
  Q_{(j)} =
 \begin{cases}
 \P(E_{2n+1-j}) \cap Q & \text{ if } 0 \leq j \leq n -1 \\
 \P(E_{2n-j}) & \text{ if } n \leq j \leq 2n-1. \\
 \end{cases}
\end{equation*}
The Schubert varieties $Q_{(j)}$
have a straightforward Bruhat ordering:

\begin{center}
\begin{tikzpicture}
  \node (a) at (.2,0) {$Q_{(2n-1)}$};
  \node (b) at (2,0) {$Q_{(2n-2)}$};
  \node (c) at (4,0) {$Q_{(n)}$};
  \node (f) at (5.5,0) {$Q_{(n-1)}$};
  \node (g) at (7.5,0) {$Q_{(1)}$};
  \node (h) at (9.5,0) {$Q_{(0)} = Q$};
  \node (pa) at (.2,.7)[teal] {$\scriptstyle{\P(E_1)}$};
  \node (pb) at (2,.7)[teal] {$\scriptstyle{\P(E_2)}$};
  \node (pc) at (4,.7)[teal] {$\scriptstyle{\P(E_n)}$};
  \node (pf) at (5.5,-.7)[teal] {$\scriptstyle{\P(E_{n+2}) \cap Q}$};
  \node (pg) at (7.5,-.7)[teal] {$\scriptstyle{\P(E_{2n}) \cap Q}$};
  
 \draw[teal, double] (a)--(pa);
 \draw[teal, double] (b)--(pb);
 \draw[teal, double] (c)--(pc);
 \draw[teal, double] (f)--(pf);
 \draw[teal, double] (g)--(pg);

 \draw[white] (a)edge node[black] {$\subset$}(b);
 \draw[white] (b)edge node[black] {$\scriptstyle{\subset \hphantom{.} \cdots \hphantom{.} \subset}$} (c);
 \draw[white] (c)edge node[black] {$\subset$}(f); 
 \draw[white] (f)edge node[black] {$\scriptstyle{\subset \hphantom{.} \cdots \hphantom{.} \subset}$}(g);
 \draw[white] (g)edge node[black] {$\subset$}(h);
\end{tikzpicture}
\end{center}

\comment{
The reason for the change in description around the middle
is that
\[
Q_{(n)} = \P(E_{n+1}) \cap Q = \P(E_n) \cap Q = \P(E_n).
\]
However, it is convenient to think
of $Q_{(n)}$ as $\P(E_n)$ since
$\iota_*[\cO_{Q_{(n)}}] =  h^{n+1}$.
}

We mention some facts about the Schubert classes in
$K(Q)$ (see \cite{buch:minuscule} for details).
For $0 \leq j \leq n-1$ we have  $[\cO_{Q_{(j)}}] = \iota^*(h^j)$.
Pushforwards of Schubert classes are given by
 \begin{equation*}
\iota_*[\cO_{Q_{(j)}}] =
 \begin{cases}
 h^j (2h-h^2) & \text{ for } 0 \leq j \leq n-1 \\
 h^{j+1} & \text{ for } n  \leq j \leq 2n-1. \\
 \end{cases}
\end{equation*}

Returning to the type $B$ Grassmannian $X$, the codimension $r$ special Schubert variety $X_{(r)}$ is defined by
\[
X_{(r)} = \{\Sigma \in X : \P(\Sigma) \cap Q_{(m+1-r)} \neq \emptyset\},
\]
for $1 \leq r \leq 2n-m$.
In other words $X_{(r)} = \pi(\psi^{-1}(Q_{(m+1-r)}))$.

We now rewrite the type $B$ triple intersection number as the sheaf
Euler characteristic of a $K(\P^{2n})$ class.  The final step
of computing Euler characteristic is exactly the same
as in Proposition \ref{P:trip_C}, and is postponed to
the unified formula in \ref{C:trip_all}.

\begin{prop}\label{P:trip_B}
Let $X := OG(m,2n+1)$.
For $T \preceq P$ in $\Omega(X)$ and $1 \leq r \leq 2n-m$, we have
\begin{equation*}
\chi_{X}([\cO_{X_P}] \cdot [\cO_{X^{T}}] \cdot [\cO_{X_{(r)}}]) =
 \begin{cases}
 \chi_{\P^{2n}}(h^{m+r+l+q-1}(2-h)^q) & \text{if } r \leq n-m, \\
 \chi_{\P^{2n}}(h^{m+r+l+q-1}(2-h)^{q-1}) & \text{if } r > n-m \text{ and } q > 0,\\
 \chi_{\P^{2n}}( h^{m+r+l-1}) = 0  & \text{if } r > n-m \text{ and } q = 0.\\
 \end{cases}
\end{equation*}
\comment{where $l$ and $q$ are the numbers of linear and quadratic
equations defining $Z_{P,T}$ as a complete intersection.}
\end{prop}

\begin{proof}
By Lemma \ref{L:cohomologically_trivial_transport}, we must simplify
$\chi_{Q}([\cO_{Z_{P,T}}] \cdot [\cO_{Q_{(m-1+r)}}])$.
In certain situations we can use the projection formula
along $\iota$ to do this.

{\bf Situation 1:}  Suppose $r \leq n-m$.
In this case,
$m-1+r \leq n-1$.
It follows that the inclusion
$\iota(Q_{(m-1+r)})$
is a complete intersection in $\P^{2n}$
cut out by $m-1+r$ linear equations and
the single quadratic equation defining $Q$.
Thus,
$[\cO_{Q_{(m-1+r)}}] = \iota^*(h^{m-1+r})$.
Using the projection formula we have
\begin{align*}
\chi_{Q}([\cO_{Z_{P,T}}] \cdot [\cO_{Q_{(m-1+r)}}]) 
&= \chi_{Q}([\cO_{Z_{P,T}}] \cdot \iota^*(h^{m-1+r})) \\
&= \chi_{\P^{2n}}([\cO_{Z_{P,T}}] \cdot h^{m-1+r}) \\
& =\chi_{\P^{2n}}( h^{l+q+m+r-1}(2-h)^q). \\
\end{align*}

{\bf Situation 2:}
Suppose $q$, the number of quadratic equations defining $Z_{P,T}$, is greater than zero.
\comment{This means $Z_{P,T}$ is not a linear subvariety of $\P^{2n}$.}
By ignoring one of the quadratic equations defining $Z_{P,T}$,
we define a larger subvariety $Z' \subset \P^{2n}$ such that
$Z_{P,T} = Z' \cap Q$.
It follows that 
$[\cO_{Z_{P,T}}] = \iota^*[\cO_{Z'}]$,
so by the projection formula we have
\begin{align*}
\chi_{Q}([\cO_{Z_{P,T}}] &\cdot [\cO_{Q_{(m-1+r)}}])
= \chi_{Q}(\iota^*[\cO_{Z'}] \cdot [\cO_{Q_{(m-1+r)}}])\\
&= \chi_{\P^{2n}}([\cO_{Z'}] \cdot \iota_*[\cO_{Q_{(m-1+r)}}])\\
&=
\begin{cases}
\chi_{\P^{2n}}(h^lh^{q-1}(2-h)^{q-1} \cdot h^{m-1+r}(2h-h^2)) & \text{if } r \leq n-m \\
\chi_{\P^{2n}}(h^lh^{q-1}(2-h)^{q-1} \cdot h^{m+r}) & \text{if } r > n-m \\
\end{cases}\\
&=
\begin{cases}
\chi_{\P^{2n}}(h^{l+q+m+r-1}(2-h)^{q}) & \text{ if } r \leq n-m \\
\chi_{\P^{2n}}(h^{l+q+m+r-1}(2-h)^{q-1}) & \text{ if } r > n-m. \\
\end{cases}
\end{align*}
Note that these situations are not mutually exclusive, and that when $r \leq n-m$ and $q > 0$, 
then both methods agree.
\comment{for calculating $\chi_{Q}([\cO_{Z_{P,T}}] \cdot [\cO_{Q_{(m-1+r)}}])$
agree. }

{\bf Situation 3:}
\comment{
Suppose $r > n-m$ and $q = 0$.  
In this case $Z_{P,T}$ is a
linear subvariety of $\P^{2n}$ that is contained in $Q$.
In other words, $[\cO_{Z_{P,T}}] = [\cO_{Q_{(j)}}]$ for some
$j \geq n$.  Similarly, since $r > n-m$,
the class $[\cO_{Q_{(m-1+r)}}]$ also
has degree greater than $n-1$.
It follows that the product
$[\cO_{Z_{P,T}}] \cdot [\cO_{Q_{(m-1+r)}}]$
is zero, since the sum of the degrees
exceeds $2n-1$,the dimension of $Q$.}
Suppose $r > n-m$ and $q = 0$.  
In this case both $Z_{P,T}$ and $Q_{(m-1+r)}$ can be thought of as
linear subvarieties of $\P^{2n}$ that are contained in $Q$.
Note that $\iota_*[\cO_{Z_{P,T}}] = h^{l}$
and  $\iota_*[\cO_{Q_{(m-1+r)}}] = h^{m+r}$,
but that
$\iota_*([\cO_{Z_{P,T}}] \cdot [\cO_{Q_{(m-1+r)}}]) = h^{m+r+l-1}$,
since one of these linear equations is redundant in the intersection
of generic translates of $Z_{P,T}$ and $Q_{(m-1+r)}$.
However, the integer $m+r+l-1 \geq 2n+1$, and therefore 
the Grothendieck class $h^{m+r+l-1} \in K(\P^{2n})$ vanishes.
\end{proof}

\subsection{Type D}
Let $X := OG(m,2n+2)$.
Let $Q := OG(1,2n+1)$ denote the quadric hypersurface of isotropic lines in $\P^{2n+1}$ with inclusion
$\iota: Q \hookrightarrow \P^{2n+1}$.
We describe the Schubert varieties relative to $E_{\bull}$ for
the even dimensional quadric $Q$.
Let $\widetilde{E}_{n+1} = \langle \+e_1, \ldots, \+e_n, \+e_{n+2}\rangle$.

The quadric $Q$ has \emph{two} Schubert varieties
of codimension $n$, defined by
\[
Q_{(n)} := \P(E_{n+1}) \text{ and }
\widetilde{Q}_{(n)} := \P(\widetilde{E}_{n+1}).
\]
For $j \neq n$ there is a single codimension $j$ Schubert variety defined by
 \begin{equation*}
  Q_{(j)} =
 \begin{cases}
 \P(E_{2n+2-j}) \cap Q & \text{ if } 0 \leq j \leq n-1 \\
 \P(E_{2n+1-j}) & \text{ if } n+1  \leq j \leq 2n. \\
 \end{cases}
\end{equation*}
The Schubert varieties in $Q$
have the following Bruhat order:
\begin{center}
\begin{tikzpicture}[sloped]
  \node (a) at (0,0) {$Q_{(2n)}$};
  \node (b) at (1.7,0) {$Q_{(2n-1)}$};
  \node (c) at (4,0) {$Q_{(n+1)}$};
  \node (d) at (5,1) {$Q_{(n)}$};
  \node (e) at (5,-1) {$\widetilde{Q}_{(n)}$};
  \node (f) at (6,0) {$Q_{(n-1)}$};
  \node (g) at (8.2,0) {$Q_{(1)}$};
  \node (h) at (10,0) {$Q_{(0)} = Q$};
  \node (pa) at (0,.7)[teal] {$\scriptstyle{\P(E_1)}$};
  \node (pb) at (1.7,.7)[teal] {$\scriptstyle{\P(E_2)}$};
  \node (pc) at (3.4,.7)[teal] {$\scriptstyle{\P(E_n)}$};
  \node (pd) at (6.3,1)[teal] {$\scriptstyle{\P(E_{n+1})}$};
  \node (pe) at (3.7,-1)[teal] {$\scriptstyle{\P(\widetilde{E}_{n+1})}$};
  \node (pf) at (6.8,-.7)[teal] {$\scriptstyle{\P(E_{n+3}) \cap Q}$};
  \node (pg) at (8.9,-.7)[teal] {$\scriptstyle{\P(E_{2n+1}) \cap Q}$};

 \draw[teal, double] (a)--(pa);
 \draw[teal, double] (b)--(pb);
 \draw[teal, double] (c)--(pc);
 \draw[teal, double] (f)--(pf);
 \draw[teal, double] (g)--(pg);
 \draw[teal, double] (d)--(pd);
 \draw[teal, double] (e)--(pe);

 \draw[white] (a)edge node[black] {$\subset$}(b);
 \draw[white] (b) edge node[black] {$\scriptstyle{\subset \hphantom{.} \cdots \hphantom{.} \subset}$} (c);
 \draw[white] (c)edge node[black] {$\subset$}(d);
 \draw[white] (c)edge node[black] {$\subset$}(e);
 \draw[white] (d)edge node[black] {$\subset$}(f);
 \draw[white] (e)edge node[black] {$\subset$}(f);
 \draw[white] (f)edge node[black] {$\scriptstyle{\subset \hphantom{.} \cdots \hphantom{.} \subset}$}(g);
 \draw[white] (g)edge node[black] {$\subset$}(h);
\end{tikzpicture}
\end{center}

We mention some facts about the Grothendieck ring
$K(Q)$, which can be found in \cite[p. 17-18]{buch:minuscule}.
As in type $B$, we have  $[\cO_{Q_{(j)}}] = \iota^*(h^j)$ for $0 \leq j \leq n-1$.
Pushforwards of Schubert classes are the same as in type $B$, the only
addition being that $\iota_*[\cO_{Q_{(n)}}] = \iota_*[\cO_{\widetilde{Q}_{(n)}}] = h^{n+1}$.
The products of codimension $n$ classes with one another depend on
the parity of $n$, in the sense that
 \begin{equation*}
[\cO_{Q_{(n)}}]^2 =
[\cO_{\widetilde{Q}_{(n)}}]^2 =
 \begin{cases}
 [\cO_{Q_{(2n)}}] = [\cO_{\text{point}}] & \text{ if } n \text{ is even,} \\
 0 & \text{ if } n \text{ is odd}, \\
 \end{cases}
\end{equation*}
whereas,
 \begin{equation*}
[\cO_{Q_{(n)}}] \cdot
[\cO_{\widetilde{Q}_{(n)}}] =
 \begin{cases}
 0 & \text{ if } n \text{ is even}, \\
 [\cO_{Q_{(2n)}}] = [\cO_{\text{point}}] & \text{ if } n \text{ is odd}. \\
 \end{cases}
\end{equation*}

The maximal even orthogonal Grassmannian, $OG(n+1,2n+2)$, has two connected components.
For any $\Sigma \in OG(n+1,2n+2)$, let $\type(\Sigma) \in \{0,1\}$
to be the codimension, mod $2$, of $\Sigma \cap E_{n+1}$ in $E_{n+1}$.
If $L \subset Q$ and $L' \subset Q$ are linear subvarieties of codimension $n$ in $Q$, then 
the affine cones $\Lambda(L)$ and $\Lambda(L')$ are elements of $OG(n+1,2n+2)$.
The $K(Q)$ classes $[\cO_L]$ and $[\cO_{L'}]$ are equal if and only if 
$\Lambda(L)$ and $\Lambda(L')$ are in the same $SO(2n+2)$ orbit,
which is the case if and only if $\type(\Lambda(L)) = \type(\Lambda(L'))$.
For any codimension $n$ linear subvariety $L \subset Q$, we let
$\type([\cO_L]) := \type(\Lambda(L))$.
It is easy to check that
$\type([\cO_{Q_{(n)}}]) = 0$ and $\type([\cO_{\widetilde{Q}_{(n)}}]) = 1$.
Hence for $\mathcal{A}$ and $\mathcal{B}$ in $\{[\cO_{Q_{(n)}}],[\cO_{\widetilde{Q}_{(n)}}]\}$,
we have
\[
 \mathcal{A} \cdot \mathcal{B} = \big( (\type(\mathcal{A}) + \type(\mathcal{B}) + n + 1) \pmod{2} \big) [\cO_{Q_{(2n)}}],
\]
where by $\pmod{2}$ we mean  ``remainder mod $2$'':
the coefficient is one if $\type(\mathcal{A}) + \type(\mathcal{B}) + n + 1$ is odd, and zero otherwise.
\comment{For the rest of this text,
we shall continue to use $\pmod{2}$ to mean ``remainder mod $2$''.}

\comment{
It is possible to calculate the product of
arbitrary Schubert classes using the facts we've mentioned,
but we will not need them here.  The reader can refer
to \cite{buch:minuscule} for further details.
}

The following lemma describes $\type([\cO_{Z_{P,T}}])$
whenever $Z_{P,T}$ has codimension $n$ in $Q$.
The lemma is adapted from the definition of the 
function $h(P,T)$ in \cite[\S 5.2]{buch:quantum_pieri}.

\begin{lemma}\label{L:projected_richardson_type}
If $Z_{P,T}$ is a linear subvariety of
codimension $n$ in the quadric $Q$,
then
 $\type(Z_{P,T}) \equiv |S| + |S'| + n + 1 \pmod{2}$, where
 \begin{align*}
 S &= \{ i \in [1, n+1] : t_j \leq i \leq p_j \text{ for some j}\}, \text{ and} \\
 S' &= \{p \in P : p \geq n+2 \text{ and } 2n+3-p \in S \}.
 \end{align*}
\end{lemma}

\begin{proof}
Let $\Lambda (Z_{P,T}) \subset \C^{2n+2}$ be the
affine cone over $Z_{P,T} \subset \P^{2n+1}$.
Note that $|S|$ is the number of $c \in [1,n+1]$ such that
$c$ is not a zero column in $D(P,T)$, and that
$|S'|$ is the number of non-zero columns $c \in [1,n+1]$ such that
column $N+1-c$ of $D(P,T)$ contains a lone star.
It follows that $|S| - |S'|$ is the number of $c \in [1,n+1]$ such that $\+e_c \in \Lambda(Z_{P,T})$,
and hence that $n+1 - (|S| - |S'|)$ is
the codimension of $\Lambda(Z_{P,T}) \cap E_{n+1}$ in $E_{n+1}$.
\end{proof}

We give an example
in which $\type([\cO_{Z_{P,T}}])$ affects a product in $K(Q)$ (which is in fact a triple intersection number).

\begin{example}
Consider $OG(2,8)$, and let $P = \{1,4\}$
and $T = \{1,2\}$.  In this case
$\type([\cO_{Z_{P,T}}]) = 0$.
It follows that
$[\cO_{Z_{P,T}}] \cdot [\cO_{Q_{(n)}}]
= (0 + 0 + 3 + 1 \mod{2}) [\cO_{Q_{(2n)}}] = 0$,
and that
$[\cO_{Z_{P,T}}] \cdot [\cO_{\widetilde{Q}_{(n)}}]
= (0 + 1 + 3 + 1 \mod{2}) [\cO_{Q_{(2n)}}] = [\cO_{Q_{(2n)}}]$.
\end{example}

Returning our attention to the 
even orthogonal Grassmannian $X$, \comment{
we note that unlike some of the Schubert varieties in $X$,
the special Schubert varieties \emph{can}
be defined by rank conditions.
In fact, the special Schubert variety $X_{(r)}$
is simply the set of elements in $X$
that intersect the affine cone over the special Schubert variety
$Q_{(m-1+r)}$ nontrivially.  In other words,}
the codimension $r$ special Schubert variety $X_{(r)}$ is defined by
$X_{(r)} := \{\Sigma \in X : \P(\Sigma) \cap Q_{(m-1+r)} \neq \emptyset\}$,
for $1 \leq r \leq 2n+1-m$.
As in the quadric case, there is an additional special Schubert variety
\footnote{
We note that our definition of the codimension $k$ special
Schubert varieties differs slightly from
the definiton in \cite[\S $3.2$]{buch:quantum_pieri}, wherein
$X_{(k)}$ and $\widetilde{X}_{(k)}$ are
switched when $n$ is odd.}
of codimension $k := n-m+1$,
defined by $\widetilde{X}_{(k)} :=  \{\Sigma \in X :  \P(\Sigma) \cap \widetilde{Q}_{(n)} \neq \emptyset\}$.
Thus $X_{(r)} = \pi(\psi^{-1}(Q_{(m-1+r)}))$ for $1 \leq r \leq 2n+1-m$ 
and $\widetilde{X}_{(k)} = \pi(\psi^{-1}(\widetilde{Q}_{(n)}))$.

Consider the triple intersection number corresponding
to $X_P$, $X^T$, and a codimension $r$ special Schubert variety.
By Lemma \ref{L:cohomologically_trivial_transport}, this equals
$\chi_Q([\cO_{Z_{P,T}}] \cdot \mathcal{A})$ where $\mathcal{A}$ is the corresponding
codimension $(m-1+r)$ special Schubert class in $K(Q)$.
We now translate this expression to the sheaf
Euler characteristic of a $K(\P^{2n+1})$ class.  The final step
of computing Euler characteristic is postponed to
the unified formula in \ref{C:trip_all}.

\comment{If needed later, $\delta$ with subscript given by $\delta^{}_{\mkern-7mu_\mathcal{A}}$.}

\begin{prop}\label{P:trip_D}
Let $X := OG(m,2n+2)$ and let $\mathcal{A}$ be a codimension $(m-1+r)$ Schubert class in $K(Q)$.
Define $\delta \in \{0,1\}$ by
\begin{equation*}
\delta \equiv 
\begin{cases}
\type(\mathcal{A}) + |S| + |S'| \pmod{2} & \text{ if } r = k, q=0, \text{ and } l = n+1, \\
1 \pmod{2} & \text{ otherwise.}\\
\end{cases}
\end{equation*}
We then have
\begin{equation*}
\chi_{Q}([\cO_{Z_{P,T}}] \cdot  \mathcal{A}) =
 \begin{cases}
 \chi_{\P^{2n+1}}(h^{m+r+l+q-1}(2-h)^q) & \text{ if } r < k \\
 \chi_{\P^{2n+1}}(h^{m+r+l+q-1}(2-h)^{q-1}) & \text{ if } r \geq k \text{ and } q > 0\\
 \chi_{\P^{2n+1}}(\delta \cdot h^{m+r+l-1})  & \text{ if } r \geq k \text{ and } q = 0.\\
 \end{cases}
\end{equation*}
\end{prop}

\begin{proof}
The proof is exactly as in type $B$, except in the case $r \geq k$ and $q=0$.
In this case, $Z_{P,T}$ is a linear subvariety of $Q$ of codimension at least $n$.
If $r > k$ or if $Z_{P,T}$ has codimension greater than $n$, then 
$[\cO_{Z_{P,T}}] \cdot \mathcal{A}=0$ (in this case $h^{m+r+l-1} \in K(\P^{2n+1})$ is also zero, since $r > n-m+1$
or $l > n+1$).
We can therefore assume $[\cO_{Z_{P,T}}]$ and $\mathcal{A}$ are both in $\{[\cO_{Q_{(n)}}],[\cO_{\widetilde{Q}_{(n)}}]\}$. 
By Lemma \ref{L:projected_richardson_type} it follows that
\begin{align*}
\chi_Q([\cO_{Z_{P,T}}] \cdot \mathcal{A})
&=  \type([\cO_{Z_{P,T}}]) + \type(\mathcal{A}) + n + 1 \pmod{2} \\
&= |S| + |S'| + \type(\mathcal{A}) \pmod{2}.
\end{align*}
This number agrees with $\chi_{\P^{2n+1}}(\delta \cdot h^{m+r+l-1})$, since $m+r+l-1 = 2n+1$.
\comment{
Equivalently, we can note that if 
the degree $(m-1+r)$ of $\mathcal{A}$  is at least $n$,
then, depending on $\delta^{}_{\mkern-7mu_\mathcal{A}}$,
there is \emph{possibly} one redundant equation
among the $m+r$ linear equations defining the variety
corresponding to $\iota_*\mathcal{A}$
and the $l$ linear equations defining $Z_{P,T}$
in $\P^{2n+1}$.
On the other hand, if the degree of 
$\mathcal{A}$  is less than $n$, then
there are $m-1+r$ linear equations
defining the variety corresponding to $\iota_*\mathcal{A}$ 
in $\P^{2n+1}$,
none of which are redundant.
Either way, we can pushforward 
to $K(\P^{2n+1})$ to get
\[
\iota_*([\cO_{Z_{P,T}}] \cdot \mathcal{A})
= \delta \cdot h^{m+r+l-1}.
\]}
\end{proof}

\subsection{A General Formula}

Propositions \ref{P:trip_C}, \ref{P:trip_B}, and \ref{P:trip_D}
are summarized concisely in
the following formulation of
the triple intersection number,
which holds for isotropic Grassmannians
of all types.

\begin{cor}\label{C:trip_all}
Let $X :=IG_{\omega}(m,N)$ be an isotropic Grassmannian, where
$N = 2n$ in type $C$, $N = 2n+1$ in type $B$, and $N=2n+2$ in type $D$.
Let $k = n-m$ in types $B$ and $C$ and let $k=n-m+1$ in type $D$.
Given $1 \leq r \leq n+k$,
suppose $\mathcal{A} \in K(IG_{\omega}(1,N))$ is a Schubert class of
codimension $(m-1+r)$, so that $\pi_*\psi^*\mathcal{A} \in K(X)$
is a special Schubert class of codimension $r$.
Given $T \preceq P$ in $\Omega(X)$, let
$l$ and $q$ be the numbers of linear and quadratic equations defining $Z_{P,T}$ 
as a complete intersection, and let $S$ and $S'$ be the sets
defined in Lemma \ref{L:projected_richardson_type}.
Define the integers $l'$, $q'$, $\eta$, and $\delta$ as follows:
\begin{align*}
q' &=
\begin{cases}
 q-1 & \text{ if $X$ is orthogonal and } q > 0, \\
 q & \text{ otherwise,}
\end{cases}\\
l' &=
\begin{cases}
 l+m+r & \text{ if $X$ is orthogonal, } q > 0 \text{ and } r \geq k, \\
 l+m+r-1 & \text{ otherwise,}
\end{cases}\\
\eta &=
\begin{cases}
 \type(\mathcal{A}) + |S| + |S'| & \text{ if } X \text{ is type D, } q = 0 \text{, and } r = k, \\
 1 & \text{ otherwise,}
\end{cases}\\
\delta &=
\begin{cases}
 0 & \text{ if } \eta \text{ is even,} \\
 1 & \text{ if } \eta \text{ is odd}.
\end{cases}
\end{align*}
We then have
\begin{align*}
\chi_{X}([\cO_{X_P}] \cdot [\cO_{X^T}] \cdot \pi_*\psi^*\mathcal{A})
 &= \chi_{\P^{N-1}}(\delta  h^{l'}(2h-h^2)^{q'})\\
 &= \delta \cdot \sum_{j = 0}^{N-1-l'-q'}{\binom{q'}{j}(-1)^{j}(2)^{q'-j}},
\end{align*}
where we define $\binom{q'}{j}$ to be zero for $j > q'$.
\end{cor}

\comment{
We note that $l'$ and $q'$ are the numbers of linear and quadratic equations
defining the projected triple intersection in $\P^{N-1}$, provided
$\delta = 1$.}

\section{Computing Pieri Coefficients}\label{S:pieris}
Let $X$ be an isotropic Grassmannian of type $B$, $C$, or $D$.
Given Schubert symbols $P$ and $Q$ and a special Schubert class $[\cO_{X_{(r)}}]$,
the $K$-theoretic structure constant $\mathcal{N}^Q_{P,r}(X)$ is the coefficient of $[\cO_{X_Q}]$
in the Pieri product $[\cO_{X_P}] \cdot [\cO_{X_{(r)}}]$.
In this section
we compute $\mathcal{N}^Q_{P,r}(X)$ as an alternating sum of triple intersection numbers
$\chi_X([\cO_{X_P}] \cdot [\cO_{X^T}] \cdot [\cO_{X_{(r)}}])$,
where $T$ ranges over a certain subset of Schubert symbols
(if $X$ is type $D$, the
special Schubert class $[\cO_{\tilde{X}_{(k)}}]$ can be substituted in order to calculate the additional Pieri coefficients
$\tilde{\mathcal{N}}^Q_{P,r}(X)$).

Given Schubert symbols $Q$ and $P$,
the M\"{o}bius function $\mu(Q,P)$  is defined by
 \begin{equation*}
\mu(Q,P) =
\begin{cases}
1 & \text{ if } Q = P, \\
- {\displaystyle \sum_{Q \preceq T \prec P}}\mu(Q,T) & \text{ if } Q \prec P, \\
0 & \text{ otherwise}.
\end{cases}
\end{equation*}
For each $Q \in \Omega(X)$, we define a class $\cO^*_Q \in K(X)$ by
 \begin{equation*}
\cO^*_{Q} := \sum_{T \in \Omega(X)}\mu(Q,T)[\cO_{X^{T}}].
\end{equation*}

\begin{lemma}\label{L:mobius_dual_class}
The class $\cO^*_Q$ is the $K$-theoretic dual
to $\cO_Q$ in the sense that
\[\chi_X([\cO_{X_P}] \cdot \cO^*_Q) = \delta_{P,Q}.\]
\end{lemma}

\begin{proof}
\comment{
 Recall that for $X$ and $Y$ in $\Omega$,
 \begin{equation*}
\mu(X,Y) =
\begin{cases}
1 & \text{ if } X = Y, \\
- {\displaystyle \sum_{X \preceq Z < Y}}\mu(X,Z) & \text{ if } X \lneq Y, \\
0 & \text{ otherwise}.
\end{cases}
\end{equation*}
}
Given Schubert symbols $T \preceq P$, the Richardson variety $Y_{P,T}$ is
rational \cite{deodhar:bruhat} with rational singularities \cite{brion:positivity}.
By \cite[p. 494]{griffiths:principles} it follows that 
\begin{equation*}
\chi_X([\cO_{X_P}] \cdot [\cO_{X^T}]) =
\begin{cases}
1 & \text{ if } T \preceq P, \\
0 & \text{ otherwise}.
\end{cases}
\end{equation*}
If $Q \prec P$, then we have
\begin{align*}
\chi_X([\cO_{X_P}] \cdot \cO^*_Q)
&= \sum_{T \in \Omega(X)} \mu(Q,T) \chi_X([\cO_{X_P}] \cdot [\cO_{X^{T}}]) \\
&= \sum_{T \preceq P}{\mu(Q,T)} \\
&= \sum_{Q \preceq T \prec P}{\mu(Q,T)} + {\mu(Q,P)}
= 0. \\
\end{align*}
If $Q \not\preceq P$, then for any $T \succeq Q$ we have $T \not\preceq P$.
Thus, $[\cO_{X_P}] \cdot [\cO_{X^{T}}] = 0$ for every Schubert class
$[\cO_{X^{T}}]$ that has nonzero coefficient in $\cO^*_Q$.
Finally, if $P = Q$ then
$\chi_X([\cO_{X_Q}] \cdot \cO^*_Q) = \chi_X([\cO_{X_Q}] \cdot [\cO_{X^{Q}}]) = 1.$
\end{proof}

Since
$\mu(Q,T) = 0$ for $Q \not\preceq T$ and $[\cO_{X_P}] \cdot [\cO_{X^T}] = 0$
for $T \not\preceq P$, we have the following corollary.

\begin{cor}\label{C:pieri_coef_as_linear_comb}
\begin{equation*}
\mathcal{N}^{Q}_{P,r}(X)
= \sum_{Q \preceq T \preceq P} \mu(Q,T)\chi_{X}([\cO_{X_P}] \cdot [\cO_{X^{T}}] \cdot [\cO_{X_{(r)}}]).
\end{equation*}
\end{cor}

It is known that $\mu(Q,T) \in \{0,(-1)^{|Q|-|T|}\}$ for any Schubert symbols $Q$ and $T$
\cite[Cor 2.7.10]{bjorner:coxeter_groups}. 
In \cite[\S A]{ravikumar:thesis} a conjectured criterion is stated for when $\mu(Q,T)$ vanishes.
We hope Corollary \ref{C:pieri_coef_as_linear_comb}
will lead to a Pieri formula for $\mathcal{N}^{Q}_{P,r}(X)$
with manifestly alternating signs, in the sense that
$(-1)^{|Q| - |P| - r} \mathcal{N}^Q_{P,r}(X) = 1$ (see \cite{brion:positivity} for a proof
of this fact).

\subsection{A Global Rule}
We briefly describe how to determine $K$-theoretic dual classes, and hence
Pieri coefficients, without a ``local'' rule, but rather using the global data 
of the entire Bruhat order.  This method requires us to invert an $L \times L$ matrix,
where $L$ is the number of Schubert symbols in $\Omega(X)$,
and allows for relatively efficient computation of $\mathcal{N}^{Q}_{P,r}(X)$.

Let $\{P_1, \ldots, P_L\}$ be the set of Schubert symbols for $X$.
\comment{Let $V \cong \C^{L}$ be the vector space underlying $K(X)$.}
Let
$\cO_i := [\cO_{X_{P_i}}]$ and 
$\cO^i := [\cO_{X^{P_i}}]$.
The sets $\{\cO_1, \ldots, \cO_L\}$ 
and  $\{\cO^1, \ldots, \cO^L\}$
are both additive bases for $K(X)$.

We will use the following four $L \times L$ matrices:
\begin{enumerate}
\item Let $\+M := (m_{ij})$ be the \emph{intersection matrix} for $X$, where
\begin{equation*}
m_{ij} = 
\begin{cases}
 1 & \text{if } P_j \preceq P_i,\\
 0 & \text{otherwise}.
\end{cases}
\end{equation*}

\comment{
Let $s_1, \ldots, s_{n+k}$ index the special Schubert classes
in the sense that \[\cO_{s_r} = \cO_{(r)}\].
}

\item Let $\+C_{(r)} := (c_{ij})$ be the \emph{Pieri coefficient matrix} for $X$,
where \[\cO_i \cdot \cO_{(r)} = c_{ij}\cO_{j}.\]

\item Let $\+T_{(r)} := (t_{ij})$ be the \emph{triple intersection matrix} for $X$,
where 
\[t_{ij} = \chi_X(\cO_i \cdot \cO^j \cdot \cO_{(r)}).\]

\item Let $\+D := \+M ^{-1}$ be the \emph{matrix of duals} on $X$.
\end{enumerate}

Let $\+d^j$ denote the $j^{\text{th}}$ column vector of $\+D$,
\comment{It represents an element of $K(X)$ with respect to the basis
$\{\cO^1, \ldots, \cO^L\}$.}
and let $\cO_{\+d^j} := \sum_{k = 1}^{L}{d_{kj}\cO^{k}}$.

\begin{obs}
The element $\cO_{\+d^j}$ is dual to $\cO_j$
in the sense that
\[\chi_X(\cO_i \cdot \cO_{\+d^j}) = \delta_{i,j}.\]
\end{obs}

\begin{proof}
 $\chi_X(\cO_i \cdot \cO_{\+d^j}) = \+m_i \cdot \+d^j$,
 where
 $\+m_i$ is the $i^{\text{th}}$ row of  $\+M$.
\end{proof}

\comment{Using the matrix of duals, we can easily convert triple
intersection numbers into Pieri coefficients.}

\begin{obs}
The matrix $\+D$ transforms triple intersection numbers
into Pieri coefficients, via the relation
\[
 \+T_{(r)} \cdot \+D = \+C_{(r)}. 
\]
\end{obs}

\begin{proof}
\begin{equation*}
\sum_{k=1}^{L}{t_{ik}d_{kj}} = \sum_{k=1}^{L}{\chi_X( d_{kj}\cO_i \cdot \cO^k \cdot \cO_{(r)})}
= \chi_X(\cO_i \cdot \cO_{\+d^j} \cdot \cO_{(r)}) = c_{ij}.
\end{equation*}
\end{proof}

\bibliographystyle{amsalpha}
\bibliography{vijay_thesis}

\end{document}